\documentclass[aop,seceqn,MSNbibl,citesort,dvips]{arximspdf}
\usepackage{graphicx}

\doi{10.1214/10-AOP583}
\volume{39}
\issue{4}
\pubyear{2011}
\firstpage{1332}
\lastpage{1364}

\makeatletter

\newtheorem{theorem}{Theorem}[section]
\newtheorem{lemma}[theorem]{Lemma}
\newtheorem{prop}[theorem]{Proposition}

\newproclaim{definition}[theorem]{Definition}
\newproclaim{remark}[theorem]{Remark}

\newcommand{\PPP}{\mathbb{P}}
\newcommand{\E}{\mathbb{E}}
\newcommand{\R}{\mathbb{R}}
\newcommand{\Z}{\mathbb{Z}}
\newcommand{\TT}{{\mathcal T}}
\newcommand{\DD}{{\mathcal D}}
\newcommand{\HH}{{\mathcal H}}
\newcommand{\JJ}{{\mathcal J}}
\newcommand{\EE}{{\mathcal E}}
\newcommand{\GG}{{\mathcal G}}
\newcommand{\KK}{{\mathcal K}}
\newcommand{\PP}{{\mathcal P}}
\newcommand{\RR}{{\mathcal R}}
\newcommand{\VV}{{\mathcal V}}

\newcommand{\bone}{{\mathbf1}}

\newcommand{\ppf}{\mathcal{H}}
\newcommand{\WW}{\mathcal{W}}
\newcommand{\NN}{\mathcal{N}}
\newcommand{\tree}{\mathbb{T}}

\newcommand{\V}{t}

\newcommand{\wh}{\widehat}
\newcommand{\wt}{\widetilde}

\newcommand{\var}{\operatorname{Var}}

\newcommand{\ol}{\overline}

\makeatother

\begin{document}
\begin{frontmatter}

\title{Markov processes on time-like graphs}
\runtitle{Markov processes on time-like graphs}

\begin{aug}
\author[A]{\fnms{Krzysztof} \snm{Burdzy}\corref{}\thanksref{t1}\ead[label=e1]{burdzy@math.washington.edu}} and
\author[A]{\fnms{Soumik} \snm{Pal}\thanksref{t2}\ead[label=e2]{soumik@math.washington.edu}}
\runauthor{K. Burdzy and S. Pal}
\affiliation{University of Washington}
\address[A]{Department of Mathematics\\
University of Washington\\
Box 354350\\
Seattle, Washington 98195\\
USA\\
\printead{e1}\\
\phantom{E-mail: }\printead*{e2}} 
\end{aug}

\thankstext{t1}{Supported in part by NSF Grant DMS-09-06743 and by
Grant N N201 397137, MNiSW, Poland.}
\thankstext{t2}{Supported partly by NSF Grant DMS-10-07563.}

\received{\smonth{12} \syear{2009}}
\revised{\smonth{4} \syear{2010}}

%
\begin{abstract}
We study Markov processes where the ``time'' parameter is
replaced by paths in a directed graph from an initial vertex to
a terminal one. Along each directed path the process is Markov
and has the same distribution as the one along any other
directed path. If two directed paths do not interact, in a
suitable sense, then the distributions of the processes on the
two paths are conditionally independent, given their values at
the common endpoint of the two paths. Conditions on graphs that
support such processes (e.g., hexagonal lattice) are
established. Next we analyze a particularly suitable family of
Markov processes, called harnesses, which includes Brownian
motion and other L\'evy processes, on such time-like graphs.
Finally we investigate continuum limits of harnesses on a
sequence of time-like graphs that admits a limit in a suitable
sense.
\end{abstract}

%
\begin{keyword}[class=AMS]
\kwd{60J60}
\kwd{60G60}
\kwd{60J99}.
\end{keyword}
\begin{keyword}
\kwd{Harness}
\kwd{graphical Markov model}
\kwd{time-like graphs}.
\end{keyword}

\end{frontmatter}

\section{Introduction}

Classical stochastic processes are families of random variables
$\{X_t, t \in\TT\}$, where $\TT$ is a subset of $\R$, for
example, positive integers or $[0,\infty)$. A notable exception
is the family of Gaussian processes, for which the structure of
the parameter space $\TT$ can be virtually arbitrary. For the
other two most popular families of stochastic processes, that
is, Markov processes and martingales, the situation is much more
complicated. The theories of Markov processes and martingales
with the parameter set $\TT$ equal to an orthant in $\R^d$ are
hard, less developed, less popular and less frequently applied
than the original theories with one-dimensional $\TT$. A~book by
Khoshnevisan \cite{K} is an excellent monograph devoted to this
field of stochastic processes.

This article introduces a class of stochastic processes where
the ``time'' parameter has been replaced by paths in a directed
graph. Our goal is to construct a time structure that matches
the Markov property better than other nonrectilinear time sets
known in the literature. A number of models with tree-like time
parameter sets have been studied, for example, branching
Brownian motion (see \cite{E}, Section~1.1), and its much more
complex version known as Le Gall's Brownian snake (see
\cite{E}, Section 3.6). Other examples include the Brownian web
and the Brownian net (see \cite{FINR,STW,SS}). In all these
models, stochastic processes are defined, in a sense, on random
graphs. In contrast, we will be concerned with a \textit{deterministic}
parameter space.

In statistics, graphical models are widely used (see
\cite{GHKM,L}) in areas such as Bayesian data analysis,
modeling of causal relationship, information retrieval and
language processing. These are probability distributions of a
countable number of random variables indexed by the vertices of
a graph where the graph structure induces a set of conditional
independence constraints (called the graphical Markov property).
One can think of our models as a continuous time analogue of
such discrete graphical models where the classical Markov
property is preserved.

Informally speaking, the following describes our set-up.
Consider the law of a classical Markov process $\PP$ in the
interval $[0,1]$. Consider a finite graph with two distinguished
vertices marked ``$0$'' and ``$1$,'' and every other vertex is
labeled by a real number between zero and one. Consider the
collection of all paths (i.e., a sequence of vertices) starting
at $0$ and ending at $1$ such that successive vertices are
increasing and share an edge in the graph. Such paths can be
seen as homeomorphic images of the unit interval $[0,1]$. Thus
every such path indexes a copy of the Markov process with law
$\PP$. We require the additional constraint that the process is
defined uniquely at every vertex.

Barring trivial example, it is not easy to even claim that such
processes exist. In fact, the existence and uniqueness of the
process depends critically on the structure of the underlying
graph. In Section~\ref{sec:TLG} we define a collection of
graphs which support such stochastic processes. We call these
time-like graphs with no co-terminal cells. In Section
\ref{sec:mkovconstruction} we construct ``natural Markov process
on a time-like graph'' and prove its uniqueness in law. In
Section~\ref{sec:BM} we provide examples of graphs that do not
satisfy our conditions and do not support a ``natural Brownian
motion.''

In the rest of the sections we focus on a class of laws $\PP$
which are called harnesses. Harnesses, defined in Section
\ref{sec:harn}, include all integrable L\'evy processes and
their corresponding bridges. The final Section~\ref{sec:hcomb},
is devoted to Brownian motion on the honeycomb graph, and its
limit as the diameter of hexagonal cells goes to zero.

\section{Time-like graphs}\label{sec:TLG}

Intuitively speaking, a time-like graph is a directed graph with
Jordan arcs as edges. We will first consider graphs with finite
numbers of vertices and edges. We will generalize our
definitions to infinite graphs at the end of Section
\ref{sec:NCC}.
\begin{definition}
A graph $\GG=(\VV,\EE)$ will be called a \textit{time-like graph}
(TLG) if its sets of vertices $\VV$ and edges $\EE$ satisfy the
following properties.

The set $\VV$ contains at least two elements, $\VV= \{\V_0,
\V_1,\ldots, \V_N\}$, where $\V_0=0$, $\V_N=1$, $\V_k \in(0,1)$
for $k=1,\ldots, N-1$ and $\V_k \leq\V_{k+1}$ for $k=0,\ldots,
N-1$. We do not exclude the case $\VV= \{\V_0=0, \V_N=1\}$.

Formally, we should say that elements of $\VV$ have the form
$(k,t_k)$, so that $(k,t_k)$ and $(k+1,t_{k+1})$ are distinct
even if $t_k = t_{k+1}$. This would make the notation very
complicated, so we will write $t_k$ instead of $(k,t_k)$. This
should not cause any confusion.

An edge between $\V_j$ and $\V_k$ will be denoted $E_{jk}$. We
assume that there is no edge between $t_j$ and $t_k$ if
$t_j=t_k$. In particular, a TLG has no loops, that is, edges of
the form $E_{jj}$. By convention, the notation $E_{jk}$
indicates that $\V_j< \V_k$. We assume that the extreme
vertices, $\V_0$ and $\V_N$, have degree 1, and all other
vertices $\V_k$ have degree~3. We assume that for every vertex
$\V_k$, except for $\V_0$ and $\V_N$, there exist edges $E_{jk}$
and $E_{kn}$ with $j< k<n$.

If there is a unique edge between $\V_j$ and $\V_k$, $j< k$,
then it will be denoted $E_{jk}$. If there are two edges between
$\V_j$ and $\V_k$, $j< k$, then they will be denoted $E'_{jk}$
and $E''_{jk}$. We will write $E_{jk}$ to refer to one of the
edges $E'_{jk}$ and $E''_{jk}$ when it is irrelevant which of
the two edges is used.
\end{definition}

See Figures~\ref{fig1},~\ref{fig2} and~\ref{fig4} below for examples
of TLGs. We will next define a ``representation'' of a TLG,
that is, a convenient geometric way to think about such a graph.
The choice of space for the representation is not significant.
We will limit ourselves to representations in $\R^3$ because it
is easy to see that every TLG has a representation in $\R^3$.
\begin{definition}\label{a11.1}
By abuse of notation, let $E_{jk}\dvtx[\V_j , \V_k] \to\R^2$
denote a continuous function. Assume that the images of the open
sets $(\V_j, \V_k)$ under the maps $t\to(t,E_{jk}(t))\in\R^3$,
where $E_{jk} \in\EE$, are disjoint. Suppose that $E_{jk}(t_k)
= E_{kn}(t_k)$ if $E_{jk}, E_{kn} \in\EE$, and $E_{jk}(t_k) =
E_{nk}(t_k)$ if $E_{jk}, E_{nk} \in\EE$. We will call the set
$\RR(\GG)=\{(t,E_{jk}(t)) \in[0,1] \times\R^2\dvtx E_{jk}\in\EE,
t\in[\V_j , \V_k]\}$ a \textit{representation} of $\GG$. We will
say that $\GG_1$ is a subgraph of $\GG_2$ and write $\GG_1 \subset
\GG_2$ if there exist representations of the two TLGs such that
$\RR(\GG_1) \subset\RR(\GG_2)$. We will call $\GG$ a
\textit{planar} TLG if it has a representation $\RR(\GG) \subset
\R^2$.
\end{definition}
\begin{remark} There are many representations for a given TLG,
but there is a unique TLG corresponding to a given
representation.
\end{remark}
\begin{definition}\label{def:2.4}
We will call a sequence of edges $(E_{k_1k_2}, E_{k_2k_3}
,\ldots,\break
E_{k_{n-1}k_n})$ a \textit{time path} if $E_{k_j, k_{j+1}}
\in\EE$ for every $j$; note that according to our conventions,
$k_1 < k_2 <\cdots< k_n$. We will write $\sigma(k_1, k_2,\ldots, k_n) $
to denote a time path $(E_{k_1k_2}, E_{k_2k_3} ,\ldots,
E_{k_{n-1}k_n})$. A time path $\sigma(k_1, k_2,\ldots,
k_n) $ will be called a \textit{full time path} if $k_1=0$ and
$k_n = N$.

Let $\ol t_j = (t_j , E_{jk}(t_j))$ for $j<N$ and $\ol t_N =
(t_N, E_{N-1,N}(t_N))$. Note that the definition does not depend
on the choice of $k$.
\end{definition}
\begin{remark} (i) Note that for every $k \in\{0,\ldots,N\}$,
there exists at least one full time path $\sigma(k_1, k_2,\ldots, k_n)
$ such that $k_m = k$ for some $m$. This follows
easily from the assumption that for every vertex $\V_k$, except
for $\V_0$ and $\V_N$, there exist edges $E_{jk}$ and $E_{kn}$
with $j< k<n$.

\mbox{}\hphantom{i}(ii) If $\GG=(\VV,\EE) \subset\GG_1=(\VV_1,\EE_1)$ and $\V_j
\in\VV_1$, then $\ol t_j \in\RR(\GG)$ does not have the same
meaning as $t_j\in\VV$.

(iii) By abuse of language, for a time path $\sigma_1=
\sigma(k_1, k_2,\ldots, k_n) $, we will call the subset $\{(t,
E_{jk}(t))\dvtx E_{jk}\in\sigma_1, t\in[\V_j , \V_k]\}$ of a
representation $\RR(\GG)$ a time path as well. If $\GG
=(\VV,\EE) \subset\GG_1=(\VV_1,\EE_1)$, $E_{jk} \in\EE_1$ and
$\{(t, E_{jk}(t)), t\in[t_j,t_k]\} \subset\RR(\GG)$, then we
will write $E_{jk} \subset\RR(\GG)$. Note that $E_{jk} \subset
\RR(\GG)$ does not imply that $E_{jk} \in\EE$.
\end{remark}

\subsection{Time-like graphs with no co-terminal cells
(NCC-graphs)} \label{sec:NCC}

We will define a subfamily of time-like graphs, with properties
that fit well with a probabilistic structure, to be presented
later. We will give two definitions of time-like graphs with no
co-terminal cells (NCC-graphs) and then we will show that the
definitions are equivalent. Each definition is more useful than
the other one in some technical arguments.
\begin{definition}\label{def:nci} (i) We will say that time
paths $\sigma(j_1, j_2,\ldots, j_{n})$ and $\sigma(k_1$, $k_2,\ldots,
k_{m})$ are \textit{co-terminal} if $j_1 = k_1$ and $j_{n}
= k_{m}$.

\mbox{}\hphantom{i}(ii) A pair of co-terminal time paths $\sigma(j_1, j_2,\ldots,
j_{n})$ and $\sigma(k_1, k_2,\ldots, k_{m})$ will be called a
\textit{cell} if $\{j_2,\ldots, j_{n-1}\} \cap\{k_2,\ldots,
k_{m-1}\} = \varnothing$. We will call $\V_{j_1}$ the \textit{start}
of the cell and $\V_{j_{n}}$ will be called the \textit{end} of
the cell.

(iii) We will call a cell $(\sigma(j_1, j_2,\ldots, j_{n}),
\sigma(k_1, k_2,\ldots, k_{m}))$ \textit{simple} if there does not
exist time path $\sigma(i_1, i_2,\ldots, i_{r}) $ such that
$i_1\in\{j_2,\ldots, j_{n-1}\}$ and $i_{r} \in\{k_2,\ldots,
k_{m-1}\}$, or $i_1\in\{k_2,\ldots, k_{m-1}\}$ and $i_{r} \in
\{j_2,\ldots, j_{n-1}\}$.

(iv) If $\V_j$ is the start of a cell, let $\V_{j^*}$ be the
smallest time $\V_k$ such that there exists a cell with the
start $\V_j$ and end $\V_k$. A cell with a start $\V_j$ and end
$\V_{j^*}$ will be called \textit{forward-minimal}. Similarly, if
$\V_k$ is the end of a cell, let $\V_{k'}$ be the largest time
$\V_j$ such that there exists a cell with the start $\V_j$ and
end $\V_k$. A cell with a start $\V_{k'}$ and end $\V_k$ will be
called \textit{backward-minimal}. We will call two cells
\textit{minimal co-terminal cells} if either they are
forward-minimal and have different starts and the same end, or
they are backward-minimal and they have different ends but the
same start.

\mbox{}\hphantom{i}(v) We will call a TLG an \textit{NCC-graph} if it does not
contain any minimal co-terminal cells.
\end{definition}
\begin{remark}\label{rem:noni} It is easy to see that if a cell
is forward-minimal or backward-minimal then it is simple. For
example, suppose that a cell $(\sigma(j_1,\break j_2,\ldots, j_{n})$,
$\sigma(k_1, k_2,\ldots, k_{m}))$ is forward-minimal,\vadjust{\goodbreak} and there is
a time path $\sigma(i_1, i_2,\ldots, i_{r}) $ such that
$i_1=j_{n_1}\in\{j_2,\ldots, j_{n-1}\}$ and $i_{r}=k_{m_1} \in
\{k_2,\ldots,\break k_{m-1}\}$. Then $\sigma(j_1,\ldots, j_{n_1},
i_2,\ldots, i_{r})$ and $ \sigma(k_1,\ldots, k_{m_1})$ form a
cell with the end $k_{m_1} < k_{m}$, contradicting the
assumption that $k_1^* = k_{m}$.
\end{remark}

The next definition, of the family of NCC$^*$-graphs, is
inductive and can be explained as follows. The simplest TLG
$\GG$, with a representation $\RR(\GG) = [0,1] \times\{0\}$, is
included in this family. If a graph $\GG_1$ already belongs to
the family of NCC$^*$-graphs, then we add a time path to
$\RR(\GG_1)$, such that the endpoints of this new path lie on a
time path already in $\RR(\GG_1)$ and neither endpoint is a
vertex already present in $\RR(\GG_1)$. Thus amended
representation corresponds to a TLG $\GG_2$ which we add to the
family of NCC$^*$-graphs.
\begin{definition}\label{def:ind}
We will define \textit{NCC$^*$-graphs} in an inductive way.

{\smallskipamount=0pt
\begin{longlist}
\item The minimal graph $\GG= (\VV, \EE)$, with $\VV= \{\V_0=0,
\V_N=1\}$ and $\EE= \{E_{0N}\}$, is an NCC$^*$-graph.

\item Suppose that a graph $\GG_1= (\VV_1, \EE_1)$ is NCC$^*$,
where $\VV_1 = \{\V_0, \V_1,\ldots,\break \V_N\}$. Suppose that $\V_j,
\V_k \notin\VV_1$, $\V_j < \V_k$, and for some $E_{j_1j_2},
E_{k_1k_2}\in\EE_1$, we have $\V_{j_1} < \V_j < \V_{j_2}$ and
$\V_{k_1} < \V_k < \V_{k_2}$. Let $\VV_2 = \VV_1 \cup\{\V_j,
\V_k\}$. Assume that there exists a time path $\sigma(m_1, m_2,\ldots,
m_n) $ such that $j_1 = m_{n_1},
j_2=m_{n_1+1},k_1=m_{n_2}$, and $k_2=m_{n_2+1} $, for some
$1\leq n_1\leq n_2 \leq n_2+1 \leq n$. If $n_1 < n_2$, then we
let $\EE_2 = (\EE_1 \cup\{E_{jk}, E_{j_1 j}, E_{jj_2},E_{k_1
k}, E_{kk_2}\}) \setminus\{E_{j_1j_2}, E_{k_1k_2}\}$. If
$n_1=n_2$, then we let $\EE_2 = (\EE_1 \cup\{E'_{jk}, E''_{jk},
E_{j_1 j}, E_{kj_2}\}) \setminus\{E_{j_1j_2}\}$. We add $\GG_2=
(\VV_2, \EE_2)$ to the family of NCC$^*$-graphs.

\item We will say that a sequence $\{\GG_j\}_{1\leq j \leq k}$
is a \textit{tower} of NCC$^*$-graphs if all graphs in the sequence
are NCC$^*$, and for every $j>1$, $\GG_j$ is constructed from
$\GG_{j-1}$ as in part (ii) of the definition.
\end{longlist}}
\end{definition}
\begin{theorem}\label{t:s25.1} \textup{(i)} A TLG is an NCC-graph if and
only if it is an NCC$^*$-graph.

\mbox{}\hphantom{i}\textup{(ii)} Every planar TLG is NCC.

\textup{(iii)} There exists a nonplanar NCC-graph.

\textup{(iv)} There exists a non-NCC-graph.
\end{theorem}
\begin{pf}
(i) \textit{Step} 1 (\textit{NCC $\Rightarrow$ NCC$^{*}$}). Suppose
that an
NCC-graph $\GG= (\VV, \EE)$ is given. We will show how to
construct it in an inductive way.

We let $\GG_1= (\VV_1, \EE_1)$ be the minimal graph with $\VV_1
= \{\V_0=0, \V_N=1\}$ and $\EE_1= \{E_{0N}\}$. Note that $\GG_1$
is an NCC$^{*}$-graph and we can find representations for $\GG$
and $\GG_1$ such that $\RR(\GG_1) \subset\RR(\GG)$.

Suppose that an NCC$^{*}$-graph $\GG_k= (\VV_k, \EE_k)$ has been
constructed and there exist representations such that
$\RR(\GG_k) \subset\RR(\GG)$. Moreover, assume that if two
edges $E_{m_1m}$ and $E_{m_2m}$ belong to $\GG_k$, then $\V_m$ is
the end of a forward-minimal cell in $\GG$. We will prove that
there exists an NCC$^{*}$-graph $\GG_{k+1} \ne\GG_k$, such that
$\RR(\GG_k) \subset\RR(\GG_{k+1}) \subset\RR(\GG)$. We will
construct $\GG_{k+1}$ in such a way that if two edges $E_{m_1m}$
and $E_{m_2m}$ belong to $\GG_{k+1}$, then $\V_m$ is the end of a
forward-minimal cell in $\GG$. Since $\GG$ has a finite number
of edges, we must have $\RR(\GG_{k+1}) = \RR(\GG)$ for some $k$,
so all we have to do to finish the proof is to complete the
inductive step.

Suppose that $\GG_k \ne\GG$. There exists $\ol t_n \in
\RR(\GG_k)$ such that $t_n \in\VV\setminus\VV_k$ and $E_{nm}
\not\subset\RR(\GG_k)$ for some $m$, because there is at least
one edge in $\RR(\GG) \setminus\RR(\GG_k)$ that is connected to
$\RR(\GG_k)$, and it is impossible for all such edges to leave
$\RR(\GG_k)$ in the negative direction. Let $\V_{j_1}$ be the
largest $\V_n$ with this property. Let $\V_{j_1^*}$ be defined
as in Definition~\ref{def:nci}(iv), relative to $\GG$.

\textit{Case} (a). Suppose that $\ol\V_{j_1^*}\in\RR(\GG_k)$. Then
there exists a time path $\sigma_1 = \sigma(q_0, q_1,\ldots,
q_{n_1})$ in $\GG$, with $q_0=j_1$, $ q_{n_1} = j_1^*$, and
$E_{q_0 q_1} \not\subset\RR(\GG_k)$.

We will now show that $\sigma_1$ does not intersect
$\RR(\GG_k)$, except for its endpoints. Suppose otherwise. Then
$\sigma_1$ intersects $\RR(\GG_k)$ at some $\ol t_m$ such that
$t_{j_1} < t_m < t_{j_1^*}$, and for some $t_{m_1}$, we have
$E_{mm_1} \in\sigma_1$ and $E_{mm_1} \not\subset\RR(\GG_k)$.
Let $t_{m_2}$ be the largest $t_m$ with these properties. Then
$\ol t_{m_2} \in\RR(\GG_k)$, $t_{m_2} \in\VV\setminus\VV_k$
and $E_{m_2m_3} \not\subset\RR(\GG_k)$ for some $m_3$. Since
$t_{j_1} < t_{m_2}$, this contradicts the definition of
$t_{j_1}$.

We add $\{(t, E_{q_rq_{r+1}}(t)), 0\leq r\leq n_1-1, t\in
[t_{q_r}, t_{q_{r+1}}]\}$ to $\RR(\GG_k)$, and we let this new
set to be the representation of $\GG_{k+1}$.

We have assumed that if two edges $E_{m_1m}$ and $E_{m_2m}$
belong to $\GG_k$, then $\V_m$ is the end of a forward-minimal
cell in $\GG$. This implies that $\V_{j_1^*}$ cannot be a vertex
of $\GG_k$ because it is the end of a forward-minimal cell in
$\GG$ which is not in $\GG_k$, and the assumption that $\GG$ is
NCC implies that there are no two forward-minimal cells in $\GG$
with the same endpoint.

The TLG $\GG_{k+1}$ is an NCC$^{*}$-graph because it was
constructed from an NCC$^{*}$-graph as in Definition
\ref{def:ind}(ii). It is clear that $\RR(\GG_{k+1}) \subset
\RR(\GG)$. Since $\V_{j_1^*}$ is the end of a forward-minimal
cell in $\GG$, all vertices in $\GG_{k+1}$ satisfy the property
that if two edges $E_{m_1m}$ and $E_{m_2m}$ belong to
$\GG_{k+1}$ then $\V_m$ is the end of a forward-minimal cell in
$\GG$.

\textit{Case} (b). Next suppose that $\ol\V_{j_1^*} \notin
\RR(\GG_k)$. The vertex $\V_{j_1^*}$ is the end of a cell
$(\sigma_3,\sigma_4)$ in $\GG$, with the start at $\V_{j_1}$.
Since $\ol\V_{j_1} \in\RR(\GG_k)$, one and only one of the time
paths $\sigma_3$ and $\sigma_4$ (say, $\sigma_3$) has an edge
$E_{j_1, j_2}$ that belongs to $\RR(\GG_k)$. Let $E_{j_3,j_4}$
be the first edge in $\sigma_3$ that does not belong to
$\RR(\GG_k)$. Then $\ol t_{j_3} \in\RR(\GG_k)$, $t_{j_3} \in
\VV\setminus\VV_k$ and $E_{j_3j_4} \not\subset\RR(\GG_k)$.
Since $t_{j_3} > t_{j_1}$, this contradicts the definition of
$t_{j_1}$. Hence, it cannot happen that $\ol\V_{j_1^*} \notin
\RR(\GG_k)$.

This completes the proof of the inductive step and shows that
NCC-graphs are NCC$^{*}$-graphs.

\textit{Step} 2 (\textit{NCC$^{*}$ $\Rightarrow$ NCC}). The proof
will be inductive. The minimal graph $\GG= (\VV, \EE)$, with
$\VV= \{\V_0=0, \V_N=1\}$ and $\EE= \{E_{0N}\}$, is an
NCC$^*$-graph, and it is also an NCC-graph.

In Definition~\ref{def:ind}, new graphs in the family of
NCC$^{*}$-graphs are created from other graphs in the same
family. Suppose that $\GG_1$ is the minimal graph defined above,
and $(\GG_1, \GG_2,\ldots, \GG_M)$ is any tower of
NCC$^{*}$-graphs.
Suppose that a vertex $\V_k$ is added to $\GG_{n-1}$ so that
$\GG_n$ is the first graph in the sequence that has the
vertex~$\V_k$. Suppose that $\V_k$ is the end of a cell. We will show
that $\V_k$ is the end of only one forward-minimal cell in
$\GG_n$, and that it will not be the end of any other
forward-minimal cell in any graph $\GG_m$ for $n\leq m \leq M$.

Note that the definitions of NCC$^*$-graphs and NCC-graphs are
invariant under time reversal, so the analysis of
forward-minimal cells can be applied to backward-minimal cells.
Hence, we will limit our argument to forward-minimal cells.

Let us recall the construction given in Definition
\ref{def:ind}(ii). Suppose that $\GG_{n-1}= (\VV_{n-1}, \EE_{n-1})$ is
NCC$^*$, where $\VV_{n-1} = \{\V_0, \V_1,\ldots, \V_N\}$. There
exist $\V_j, \V_k \notin\VV_{n-1}$, $\V_j < \V_k$ such that
$E_{j_1j_2}, E_{k_1k_2}\in\EE_1$ for some $\V_{j_1} < \V_j <
\V_{j_2}$ and $\V_{k_1} < \V_k < \V_{k_2}$. We have $\GG_{n}=
(\VV_{n}, \EE_{n})$, $\VV_n = \VV_{n-1} \cup\{\V_j, \V_k\}$.
There exists a time path $\sigma(m_1, m_2,\ldots, m_n) $ such
that $j_1 = m_{n_1}, j_2=m_{n_1+1},k_1=m_{n_2}$, and
$k_2=m_{n_2+1} $, for some $1\leq n_1\leq n_2 \leq n_2+1 \leq
n$. There are two possible cases: (a) If $n_1 < n_2$, then $\EE_2
= (\EE_1 \cup\{E_{jk}, E_{j_1 j}, E_{jj_2},E_{k_1 k},
E_{kk_2}\}) \setminus\{E_{j_1j_2}, E_{k_1k_2}\}$; (b) if
$n_1=n_2$, then $\EE_2 = (\EE_1 \cup\{E'_{jk}, E''_{jk}, E_{j_1
j}, E_{kj_2}\}) \setminus\{E_{j_1j_2}\}$.\vspace*{1pt}

First we will show that $\V_k$ is the end of only one
forward-minimal cell in~$\GG_n$. In case (b), it is obvious that
there is only one cell $(\{E'_{jk}\},\{ E''_{jk}\})$ that is
forward-minimal and has the end at $\V_k$. Consider case (a) and
let $\sigma_1 = \sigma(j, m_{n_1+1},\ldots,m_{n_2},k)$. The cell
$(\sigma_1, \{E_{jk}\})$ is forward-minimal and has the end at
$\V_k$. Suppose that some other cell $(\sigma_2, \sigma_3)$ in
$\GG_n$ has $\V_k$ as its end, and call its start~$\V_r$. Then
one of the time paths $\sigma_2$ or $ \sigma_3$ (say,
$\sigma_2$) must pass through $\V_j$, and the other one,
$\sigma_3$, must pass through $\V_{k_1}$. Recall that $\sigma_1$
is a time path that goes through $\V_j$ and $\V_{k_1}$. Let the
concatenation of the part of $\sigma_2$ between $\V_r$ and
$\V_j$ and the part of $\sigma_1$ between $\V_j$ and $\V_{k_1}$
be called $\sigma_4$. The time path $\sigma_4$ starts at $\V_r$
with an edge different from the first edge of $\sigma_3$. The
paths $\sigma_3$ and $\sigma_4$ contain $\V_{k_1}$ so a
forward-minimal cell with start $\V_r$ must have the end at
$\V_{k_1}$ or an earlier time. Therefore, it cannot have the end
at $\V_k> \V_{k_1}$.

Next we will show that $\V_k$ cannot be the end of two
forward-minimal cells in any graph $\GG_m$, $m>n$. Suppose to
the contrary that $\V_k$ is the end of two different
forward-minimal cells in $\GG_q$ for some $q>n$, but it is not
the end of two different forward-minimal cells in $\GG_m$,
$m<q$. Recall that, according to our construction, $\V_k =
\V_{j^*}$, that is, when we added $\V_k$ to the set of vertices,
we also created a forward-minimal cell with start $\V_j$ and end
$\V_k$. Suppose that $\GG_q$ was constructed by adding an edge
$E_{\ell_1\ell_2}$ to $\GG_{q-1}$, and this procedure created a
new forward-minimal cell $(\sigma_1,\sigma_2)=(\sigma(q_1,\ldots,
q_{s_1}, k), \sigma(r_1,\ldots, r_{s_2}, k))$ in $\GG_q$
with $q_1=r_1 \ne j$. The edge $E_{\ell_1\ell_2}$ must belong to
one of the time paths in this cell, say, $\ell_1 = q_{s_3}$,
$\ell_2 = q_{s_3+1}$ and $q_{s_3+1} \ne k$. According to
Definition~\ref{def:ind}(ii), $\GG_{q-1}$ must contain either a
time path $\sigma(q_{s_3-1}, u_1,\ldots, u_{s_4}, q_{s_3+2})$ or
$\sigma(q_{s_3-1}, q_{s_3+2})$. Then $\GG_{q-1}$ contains the
cell $(\sigma(q_1,\ldots, q_{s_3-1}, u_1,\ldots, u_{s_4},
q_{s_3+2},\ldots, q_{s_1}, k), \sigma(r_1,\ldots, r_{s_2}, k))$,
possibly with $u_1,\ldots, u_{s_4}$ missing in the first path.
If this is a forward-minimal cell, then this contradicts the
assumption that there is only one forward-minimal cell in
$\GG_{q-1}$ with end $\V_k$. If this cell is not forward-minimal,
then $q_1^*$, defined as in Definition~\ref{def:nci}(iv),
satisfies $q_1^* < k$, relative to $\GG_{q-1}$. This implies
that $q_1^* < k$, relative to~$\GG_{q}$, which contradicts the
assumption that $(\sigma_1,\sigma_2)$ is a forward minimal cell.
This completes the proof of part (i).

\mbox{}\hphantom{i}(ii) It is easy to see that if a TLG $\GG$ is planar, then the
region enclosed by $\RR(\GG)$ is divided by $\RR(\GG)$ into
nonintersecting cells that are both forward-minimal and
backward-minimal. Therefore every vertex, except $\V_0$ and
$\V_N$, is either the start or the end of a single cell that is
forward-minimal and backward-minimal.

(iii) Let $\GG= (\VV, \EE)$, where $\VV= \{\V_j =j/7, j=0,1,\ldots, 7\}
$ and
\[
\EE= \{ E_{0, 1 }, E_{1 ,2 },
E_{2 ,3 }, E_{3 ,4 }, E_{4 ,5 }, E_{5 ,6 }, E_{6 ,7}, E_{1 ,4 },
E_{2 ,5 }, E_{3 ,6 }\}.
\]
It is elementary to check
that $\GG$ is an NCC-graph and that it is not planar. See
Figure~\ref{fig1}.

%
\begin{figure}

\includegraphics{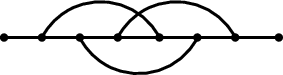}

\caption{A nonplanar NCC-graph.}\label{fig1}
\end{figure}

(iv) Let $\GG= (\VV, \EE)$, where $\VV= \{\V_j =j/7, j=0,1,\ldots, 7\}$ and
\[
\EE= \{ E_{0, 1 }, E_{1 ,2 },
E_{1 ,3 }, E_{2 ,4 }, E_{2 ,5 }, E_{3 ,4 }, E_{3 ,5 }, E_{4 ,6
}, E_{5 ,6 }, E_{6 ,7}\}.
\]
The cells
$(\sigma(3,4,6), \sigma(3,5,6))$ and $(\sigma(2,5,6),
\sigma(2,4,6))$ are minimal and co-terminal. Hence, $\GG$ is not
%
%
\begin{figure}[b]

\includegraphics{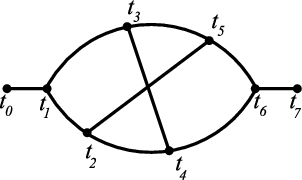}

\caption{A non-NCC-graph.}\label{fig2}
\end{figure}
an NCC-graph. See Figure~\ref{fig2}.
\end{pf}
\begin{remark}
The analysis of NCC-graphs is somewhat complicated due to the
following facts.

\mbox{}\hphantom{i}(i) If $\GG_1$ and $\GG_2$ are NCC-graphs and $\RR(\GG_1)
\subset\RR(\GG_2)$, then it does not necessarily follow that
$\GG_1$ and $\GG_2$ belong to a tower of NCC-graphs. For example,
let $\GG_1$ be obtained from the graph in Figure~\ref{fig2} by
removing $E_{34}$ and $E_{25}$, and let $\GG_2$ be obtained from
the graph in Figure~\ref{fig2} by removing $E_{34}$. Adding
$E_{25}$ to $\GG_1$ does not conform to the rules of Definition
\ref{def:ind}.

(ii) It is quite obvious that there exist TLGs $\GG_1$ and
$\GG_2$ such that $\RR(\GG_1) \subset\RR(\GG_2)$, $\GG_1$ is
NCC and $\GG_2$ is not NCC. For example, take $\GG_1$ to be a
single full path and $\GG_2$ to be the graph in Figure~\ref{fig2}.
It is less obvious that there exist TLGs $\GG_1$ and~$\GG_2$
such that $\RR(\GG_1) \subset\RR(\GG_2)$, $\GG_2$ is NCC and
$\GG_1$ is not NCC. For example, let $\GG_2$ be the graph in
Figure~\ref{fig4}. To see that $\GG_2$ is NCC, note that one can
%
%
\begin{figure}

\includegraphics{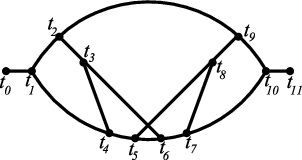}

\caption{An NCC-graph that contains a non-NCC-graph.}\label{fig4}
\end{figure}
construct it as in Definition~\ref{def:ind} by starting with the
full path $\sigma( t_0,t_1, t_4, t_5, t_6, t_7, t_{10}, t_{11})$
and adding edges in this order: $\sigma(t_1,t_2,t_3,t_4)$,
$\sigma(t_3,t_6)$, $\sigma(t_7,t_8,t_9,t_{10})$,
$\sigma(t_5,t_8)$, $\sigma(t_2,t_9)$. Let $\GG_1$ be the graph
obtained by removing $E_{34}$ and $E_{78}$ from $\GG_2$. The
graph $\GG_1$ is topologically the same as that in
Figure~\ref{fig2} so it is non-NCC.
\end{remark}

We will extend the definition of TLGs to graphs with infinitely
many vertices. First, we present two simple generalizations of
TLGs with finite $\VV$.
It will be convenient to allow TLGs (with finitely many
vertices) in which $t_0$ and $t_N$ take values in $\R\cup
\{-\infty, \infty\}$ with the restriction that $t_0 < t_N$.
Clearly, all theorems proved so far apply to thus enlarged
family of TLGs. Note that allowing $t_0$ and $t_N$ to take
infinite values does not add anything significant to the model
because we can rescale the graph by the
deterministic function $t \to\arctan t$. We allow for infinite
values of $t_0$ and $t_N$ to be able to study standard examples
of Markov processes on the real line.
\begin{definition}\label{m27.1}
(i) Suppose that the vertex set of a graph $\GG= (\VV,\EE)$
is infinite. We will call $\GG$ a \textit{time-like graph} (TLG) if
it satisfies the following conditions. (a) There exists a
sequence of TLGs $\GG_n = (\VV_n, \EE_n)$, $n\geq1$, such that
each $\VV_n$ is finite, and for some representations of
$\GG_n$'s and $\GG$ we have $\RR(\GG_n) \subset\RR(\GG_{n+1})$
for every $n$, and $\bigcup_n \RR(\GG_n) = \RR(\GG)$. (b) The
graph $\GG$ is \textit{locally finite}, that is, it has a
representation $\RR(\GG) \subset\R^3$ such that for any compact
set $K \subset\R^3$, only a finite number of edges intersect
$K$.

(ii) We will call a TLG $\GG$ with infinite vertex set an
NCC-graph if it satisfies the following conditions. (a) The
sequence $\{\GG_n\}_{n\geq1}$ in part (i) of the definition can
be chosen so that it is a tower of NCC-graphs in the sense of
Definition~\ref{def:ind}(iii). (b)~Let $\VV_n = \{ t_{0,n},
t_{1,n},\ldots, t_{N_n, n} \}$. The initial vertices
$t_{0,n}\in\VV_n$ and terminal vertices $t_{N_n,n}\in\VV_n$
are the same for all $\GG_n$, that is, $t_{0,n} = t_{0,m}$ and
$t_{N_n,n} = t_{N_m,m}$ for all $n$ and $m$.
\end{definition}
\begin{remark}
(i) Recall the notation from Definition~\ref{m27.1}(ii). It follows
from conditions (a) and (b) of that definition that
the initial
edges form a decreasing sequence, that is, $E_{t_{0,n}, t_{1,n}}
\subset E_{t_{0,m}, t_{1,m}}$ if $n>m$. Similarly, terminal
edges form a decreasing sequence, that is, $E(t_{N_n-1,n},
t_{N_n,n}) \subset E(t_{N_m-1,m}, t_{N_m,m})$ if $n>m$.

(ii) It is easy to see that if $\GG= (\VV,\EE)$ is a
TLG with infinite number of vertices, then all vertices have
degree 3, except for at most two vertices with degree 1.
\end{remark}

\section{Markov processes on time-like
graphs}\label{sec:mkovconstruction}

Suppose that $\GG=(\VV,\EE)$ is a TLG and $\VV$ is finite. Let $\PP
$ denote
the distribution of a Markov process $\{Y(t)$, $t\in[t_0,t_N]\}$.
We do not assume that the Markov process is necessarily
time-homogeneous, that is, that its transition probabilities are
invariant under time shifts.

The regular conditional distribution of $\{Y(t), t\in
[t_1,t_2]\}$ given $\{Y(t_1)=y_1$, $Y(t_2)=y_2\}$ exists for
$\PP$-almost all values of $(Y(t_1), Y(t_2))$, under mild
assumptions on the state space of $Y$ (see Section 21.4 in
\cite{FG} for a discussion of conditional probabilities). The
conditional distribution of $\{Y(t), t\in[t_1,t_2]\}$ given
$\{Y(t_1)=y_1, Y(t_2)=y_2\}$ will be called a \textit{Markov
bridge}. The Markov bridge is a (time-inhomogeneous) Markov
process on the interval $[t_1, t_2]$.
\begin{definition}
Let $X$ be a collection of random variables $X_{jk}(t)$, for all
$E_{jk} \in\EE$ and $t\in[\V_j, \V_k]$. If $E_{jk}, E_{kn}\in
\EE$, then we assume that $X_{jk}(\V_k) = X_{kn}(\V_k)$, and
similarly, if $E_{jk}, E_{nk}\in\EE$, then $X_{jk}(\V_k) =
X_{nk}(\V_k)$.

Recall that we may have two edges $E'_{jk}$ and $E''_{jk}$ with
the same endpoints $\V_j$ and $\V_k$. Then the collection of
random variables $X_{jk}(t)$ contains separate families
$\{X'_{jk}(t),t\in[\V_j, \V_k]\}$ and $\{X''_{jk}(t),t\in
[\V_j, \V_k]\}$ corresponding to each edge.\vspace*{1pt}

Consider a time path $\sigma_1 = \sigma(k_1, k_2,\ldots, k_n) $
and let $X_{\sigma_1} (t) = X_{k_1, k_2,\ldots, k_n}(t) =
X_{k_j, k_{j+1}}(t)$ for all $j = 1, 2,\ldots, n-1$ and $t\in
[\V_{k_j}, \V_{k_{j+1}}]$. We will call $X$ a
\textit{$\PP$-process on $\GG$} if for every full time path
$\sigma$, the process $\{X_\sigma(t), t\in[t_0,t_N]\}$ has
distribution $\PP$. We will write $X(t)$ instead of $X_{jk}(t)$ or
$X_\sigma(t)$ when no confusion may arise.
\end{definition}

We extend the notion of a $\PP$-process on a TLG (with
finite $\VV$) to processes that are defined for all $t\in E,
E\in\EE$, except $t_0$ and $t_N$. For example, we can take $t_0
= -\infty$, $t_N =\infty$ and let $\PP$ be the distribution of a
two-sided Brownian motion conditioned to have value 0 at time 0.
This extension does not pose any technical problems
but allows us to consider natural examples.

Note that if $X$ is a $\PP$-process and $\sigma_1 = \sigma(k_1,
k_2,\ldots, k_n) $ then conditionally on $X(t_{k_j}) = x_j$, $1\leq
j \leq n$, the path $\{X(t), t\in\sigma_1\}$ has the same
distribution as the concatenation of independent Markov bridges
from $(t_{k_j}, X(t_{k_j}))$ to $(t_{k_{j+1}}$, $X(t_{k_{j+1}}))$,
$1\leq j \leq n$.

For every TLG $\GG$ and every $\PP$, there exists a
$\PP$-process on $\GG$. A trivial example of a $\PP$-process on
a TLG can be constructed by taking a Markov process $\{Y(t),
t\in[t_0,t_N]\}$ with distribution $\PP$ and then letting
$X_{jk}(t) = Y(t)$ for all $E_{jk}\in\EE$ and $t\in[\V_{j},
\V_{k}]$.
\begin{definition}\label{m23.1}
Suppose that $\WW\subset\RR(\GG)$ is a
finite nonempty set such that $\RR(\GG) \setminus\WW$ is
disconnected. Some edges of $\GG$ are cut by $\WW$ into two or
more sub-edges; let us call this new collection of edges
$\EE_0$. Suppose that $\EE_1$ and $\EE_2$ are disjoint sets of
edges with the union equal to $\EE_0$. Each set $\EE_1$ and
$\EE_2$ may consist of several connected components of $\RR(\GG)
\setminus\WW$. We will call a process $X$ on a TLG $\GG$ a
\textit{graph-Markovian} process if for all $\WW, \EE_1$ and
$\EE_2$, the conditional distribution of $\{X(t), t\in E, E\in
\EE_1\}$ given $\{X(t), t\in E, E\in\EE_2\}$ depends only on
$\{X(t), t\in\WW\}$.
\end{definition}
\begin{definition}\label{m23.2}
For a point $t$ in $\GG$, let $F(t)$
(``the future of $t$'') be the set of all points $s \geq t$ such
that there is a full path passing through $t$ and $s$.
Similarly, let $P(t)$ (``the past of $t$'') be the set of all
points $s \leq t$ such that there is a full path passing through
$t$ and $s$. We will say that a process $X$ on a TLG $\GG$ is
\textit{time-Markovian} if for every $t$, the conditional
distributions of $\{X(s), s\in F(t)\}$ and $\{X(s), s\in P(t)\}$ given
$X(t)$ are independent.
\end{definition}
\begin{remark}\label{m22.1}
(i) Suppose that a process $X$ on a TLG $\GG$ is time-Markovian, and
$s$ and $t$ lie on a full path $\sigma$, with $s<t$. It is easy
to see that the conditional distributions of $\{X(u), s\leq
u\leq t, u \in\sigma\}$, $\{X(u), u\in F(t)\}$ and $\{X(u),
u\in P(s)\}$ given $X(s)$ and $X(t)$ are jointly independent. Moreover,
the conditional distribution of $\{X(u), s\leq u\leq t, u \in
\sigma\}$ given $X(s)$ and $X(t)$ is that of a Markov bridge
between $(s, X(s))$ and $(t, X(t))$.

(ii) It is easy to see that if a process $X$ on a TLG $\GG$ is
graph-Markovian, then the families of random variables $\{X(t), t\in
E\}$, $E\in\EE$, are conditionally independent given $\{X(t),
t\in\VV\}$, and for every $E_{jk}\in\EE$, the conditional
distribution of $\{X(t), t\in E_{jk}\}$ given $\{X(t), t\in
\VV\}$ is a Markov bridge between $(t_j, X(t_j))$ and $(t_k,
X(t_k))$.
\end{remark}
\begin{definition}
We will say that a $\PP$-process $X$ on a TLG $\GG$ with finite
vertex set $\VV$ is \textit{natural} if it is time-Markovian and
graph-Markovian.
\end{definition}

Recall that we call a cell $(\sigma(j_1, j_2,\ldots, j_{n_1}),
\sigma(k_1, k_2,\ldots, k_{n_2}))$ simple if there is no time
path $\sigma(m_1, m_2,\ldots, m_{n_3}) $ such that we have
$m_1\in\{j_2,\ldots,\break j_{n_1-1}\}$ and $m_{n_3} \in\{k_2,\ldots,
k_{n_2-1}\}$, or $m_1\in\{k_2,\ldots, k_{n_2-1}\}$ and
$m_{n_3} \in\{j_2,\ldots,\break j_{n_1-1}\}$.
\begin{definition}
We will say that a process $X$ on a TLG $\GG$ is
\textit{cell-Markovian} if for any simple cell consisting of
$\sigma(j_1, j_2,\ldots, j_{n_1})$ and $\sigma(k_1, k_2,\ldots,\break
k_{n_2})$, the processes $\{X_{j_1, j_2,\ldots, j_{n_1}}(t),
t\in[\V_{j_1}, \V_{j_{n_1}}]\}$ and $\{X_{k_1, k_2,\ldots,
k_{n_2}}(t), t\in[\V_{k_1},\break \V_{k_{n_2}}]\}$ are conditionally
independent, given the values of $X_{j_1, j_2,\ldots,
j_{n_1}}(\V_{j_1})$ and $X_{j_1, j_2,\ldots,
j_{n_1}}(\V_{j_{n_1}})$ [these are the same as $X_{k_1, k_2,\ldots,
k_{n_2}}(\V_{k_1})$ and $X_{k_1, k_2,\ldots,
k_{n_2}}(\V_{k_{n_2}})$].
\end{definition}

Note that there is no direct logical relation between the
notions of time-Markovian, graph-Markovian and cell-Markovian
processes.
\begin{theorem}\label{th:nat} \textup{(i)} For every NCC-graph $\GG$ with
finite vertex set $\VV$ and every Markov process $\PP$, there
exists a natural $\PP$-process $X$ on $\GG$, and the
distribution of such a process is unique. The natural
$\PP$-process is cell-Markovian.

\textup{(ii)} Suppose that for some TLG $\GG$ with $\VV= \{t_0=0, t_1,\ldots,
t_N=1\}$, there exist simple
coterminal cells $(\sigma_1,\sigma_2)$ with endpoints $t_1<t_2$,
and $(\sigma_3,\sigma_4)$ with endpoints $t_3<t_4$. Assume that
either $ t_1 < t_3$ or $ t_2 < t_4$. Then there is no natural
Brownian motion on $\GG$.
\end{theorem}
\begin{remark}\label{m25.1} (i) Part (ii) of Theorem
\ref{th:nat} cannot be generalized to say that ``Then there is
no natural Markov process on $\GG$.'' The reason is that the
process identically equal to 0 is a natural Markov processes on
every TLG. There are also less trivial examples.

\mbox{}\hphantom{i}(ii) If we take $t_2 = 3/7$ instead of $2/7$ in Figure
\ref{fig2}, then we will have an example of a TLG with
coterminal cells for which neither
$t_1 < t_3$ nor $ t_2 < t_4$ holds. The starts $t_2$ and $t_3$
of the two cells correspond to the same time $3/7$.

(iii) If $\GG_1 \subset\GG_2$, both graphs are NCC, $X$ is a
natural $\PP$-process on $\GG_1$ and $X'$ is a natural
$\PP$-process on $\GG_2$, then it is not necessarily true that
the distribution of $X$ is that of $X'$ restricted to $\GG_1$.
To see this, let $\PP$ be the distribution of Brownian motion,
$\GG_2$ be the graph in Figure~\ref{fig4} and let $\GG_1$ be the
graph obtained by deleting the edges $E_{36}, E_{23}$ and
$E_{34}$. One can check that the joint distribution of $(X(t_2),
X(t_4))$ is different from that of $(X'(t_2), X'(t_4))$. This
can be shown by applying Proposition~\ref{cov} to $(X(t_2),
X(t_4))$. To determine the distribution of $(X'(t_2), X'(t_4))$,
note that $t_2$ and $t_4$ lie on a full time path in $\GG_2$.

(iv) Suppose that an NCC-graph $\GG$ is the last element of a
tower of NCC-graphs $\{\GG_k\}_{1\leq k \leq n}$. Then the
restriction of a natural process $X$ on $\GG$ to any $\GG_j$,
$1\leq j \leq n$, is a natural process on $\GG_j$. This follows
from the proof of Theorem~\ref{th:nat} below and from the
uniqueness of the natural $\PP$-process.

\mbox{}\hphantom{i}(v) Does uniqueness in Theorem~\ref{th:nat}(i) hold true if we replace
``natural'' with ``graph-Markovian?'' We leave this as an open problem.
\end{remark}

We will prove part (i) of Theorem~\ref{th:nat} in this section
and part (ii) in the next section.
\begin{pf*}{Proof of Theorem~\ref{th:nat}\textup{(i)}}
We assume that $\VV$ is finite, $t_0 = 0$ and $t_N=1$.
Fix any Markov process distribution $\PP$.
We will use induction, since according to Theorem~\ref{t:s25.1},
the family of all NCC-graphs can be constructed inductively, as
in Definition~\ref{def:ind}.

It is obvious that there exists a unique in law natural
$\PP$-process on the minimal graph $\GG= (\VV, \EE)$, with $\VV
= \{\V_1=0, \V_N=1\}$ and $\EE= \{E_{1N}\}$. It is also easy to
see that this process is cell-Markovian.

Suppose that $\GG_1$ is an NCC-graph. We make the inductive
assumption that there exists a natural $\PP$-process $X$ on
$\GG_1$, it is unique in law and it is cell-Markovian.

Recall how a new graph $\GG_2$ is constructed in part (ii) of
Definition~\ref{def:ind}. Suppose that $\GG_1= (\VV_1, \EE_1)$
with $\VV_1 = \{\V_1, \V_2,\ldots, \V_N\}$. Suppose that $\V_j,
\V_k \notin\VV_1$, $\V_j < \V_k$, and for some $E_{j_1j_2},
E_{k_1k_2}\in\EE_1$, we have $\V_{j_1} < \V_j < \V_{j_2}$ and
$\V_{k_1} < \V_k < \V_{k_2}$. Let $\VV_2 = \VV_1 \cup\{\V_j,
\V_k\}$. Assume that there exists a time path $\sigma(m_1,
m_2,\ldots,\break
m_n) $ such that $j_1 = m_{n_1},
j_2=m_{n_1+1},k_1=m_{n_2}$, and $k_2=m_{n_2+1} $, for some
$1\leq n_1\leq n_2 \leq n_2+1 \leq n$. (a) If $n_1 < n_2$, then
we let $\EE_2 = (\EE_1 \cup\{E_{jk}, E_{j_1 j}, E_{jj_2},\vspace*{1pt}\break E_{k_1
k}, E_{kk_2}\}) \setminus\{E_{j_1j_2}, E_{k_1k_2}\}$. (b) If
$n_1=n_2$, then we let $\EE_2 = (\EE_1 \cup\{E'_{jk}, E''_{jk},\break
E_{j_1 j}, E_{kj_2}\}) \setminus\{E_{j_1j_2}\}$. Then $\GG_2=
(\VV_2, \EE_2)$.

It will suffice to show that there exists a natural
$\PP$-process $X$ on $\GG_2$; it is unique in law and it is
cell-Markovian.

In case (a), we effectively add only one edge $E_{jk}$ to graph
$\GG_1$. Other ``new'' edges $E_{j_1 j}, E_{jj_2},E_{k_1 k}$ and
$E_{kk_2}$ are created by subdividing $E_{j_1j_2}$ and
$E_{k_1k_2}$. Let $Z_1 = X_{j_1j_2} (\V_j)$ and $Z_2 =
X_{k_1k_2} (\V_k)$. We define $\{X'_{jk}(t),t\in[\V_j, \V_k]\}$
to be a Markov bridge between $(\V_j, Z_1)$ and $(\V_k, Z_2)$,
otherwise independent of $\{X(t), t \in\GG_1\}$. In other
words, $X'_{jk}(\V_j)=Z_1$, $X'_{jk}(\V_k)=Z_2$, and the
distribution of $\{X'_{jk}(t),t\in[\V_j, \V_k]\}$ is the same
as that of the process $\{Y(t) , t\in[\V_j, \V_k]\}$ under
$\PP$, conditioned by $Y(\V_j) = Z_1$ and $Y(\V_k) = Z_2$. We
define a process $X'$ on TLG $\GG_2$ by letting it have the same
values as $X$ on $\RR(\GG_1)$, and using the above definition on
$E_{jk}$.

In case (b), we add two edges $E'_{jk}$ and $ E''_{jk}$ to
$\GG_1$. The other new edges $E_{j_1 j}$ and $E_{kj_2}$ are
created by subdividing $E_{j_1j_2}$. Let $Z_1 = X_{j_1j_2}
(\V_j)$ and $Z_2 = X_{j_1j_2} (\V_k)$. We define
$\{X'_{jk}(t),t\in[\V_j, \V_k]\}$ and $\{X''_{jk}(t),t\in
[\V_j, \V_k]\}$ to be independent Markov bridges between $(\V_j,
Z_1)$ and $(\V_k, Z_2)$, otherwise independent of $\{X(t), t \in
\GG_1\}$. We choose the representations $\RR(\GG_1) $ and $
\RR(\GG_2)$ so that they agree on $\RR(\GG_1) $ with the part of
$E_{j_1,j_2}$ between $t_j$ and $t_k$ removed. We define a
process $X'$ on TLG $\GG_2$ by first letting it have the same
values as $X$ on $\RR(\GG_1) \setminus E_{jk}$. The process
$\{X'_{jk}(t),t\in[\V_j, \V_k]\}$ represents the values of $X'$
on the edge $E'_{jk}$ and $\{X''_{jk}(t),t\in[\V_j, \V_k]\}$
represents the values of $X'$ on the edge $E''_{jk}$.

In the rest of the proof, we will focus on case (a). Case (b)
requires minor modifications and is left to the reader.

Recall that $\GG_1$ contains a time path $\sigma(m_1, m_2,\ldots, m_n)
$ with $j_1 = m_{n_1}, j_2=m_{n_1+1},k_1=m_{n_2}$
and $k_2=m_{n_2+1} $, for some $n_1 < n_2$. This implies that
$\GG_2$ must contain a time path $\sigma(m_1,\ldots, j_1, j,
j_2,\ldots, k_1, k , k_2,\ldots, m_n) $. There is a Markov
bridge between $(\V_j, Z_1)$ and $(\V_k, Z_2)$ in the
representation $\RR(\GG_1)$. The construction of
$\{X'_{jk}(t),t\in[\V_j, \V_k]\}$ consists of generating an
independent Markov bridge between the same points. By the Markov
property\vspace*{1pt} of $\PP$, the distribution of $X'_{j_1, j, k , k_2}$ on
the graph $\GG_2$ is the same as the distribution of $X_{j_1,
j_2,\ldots, k_1, k_2}$ on the graph $\GG_1$. This implies that
for every full path $\sigma(r_1,\ldots,j_1, j, k , k_2,\ldots,
r_n)$ in $\GG_2$, the distribution of $X'_{r_1,\ldots,j_1, j, k
, k_2,\ldots, r_n}$ is $\PP$. Hence, $X'$ is a $\PP$-process on
$\GG_2$.

Next we will show that $X'$ is cell-Markovian. Consider a simple
cell $(\sigma_1, \sigma_2)$ in $\GG_2$. Suppose that the paths
$\sigma_1$ and $ \sigma_2$ do not contain the new edge $E_{jk}$.
Then $(\sigma_1, \sigma_2)$ is a simple cell in $\GG_1$. By the
inductive assumption, the processes $X$ on $\sigma_1$ and $X$ on
$\sigma_2$ are conditionally independent given their values at
the end and start of the cell. Since $X'$ is equal to $X$ on
$(\sigma_1, \sigma_2)$, the same claim holds for $X'$.

Now consider a simple cell $(\sigma_1, \sigma_2)$ in $\GG_2$
such that $\sigma_1$ contains $E_{jk}$. Then we have $\sigma_1=
\sigma(r_1,\ldots, j_1, j, k , k_2,\ldots, r_{n_1})$ and
$\sigma_2= \sigma(q_1,\ldots, q_{n_2})$. We will show that
processes $X'_{r_1,\ldots, j_1, j, k , k_2,\ldots, r_{n_1}}$ and
$X'_{q_1,\ldots, q_{n_2}}$ are conditionally independent given
their values at the start and end.

First, we will argue that the cell consisting of time paths $\sigma
_3=\sigma(r_1,\ldots, j_1$, $j_2,\ldots, k_1, k_2,\ldots, r_{n_1})$ and
$\sigma_4
= \sigma(q_1,\ldots, q_{n_2})$ is simple in $\GG_1$. Suppose
otherwise, that is, there exists a time path $\sigma_5 =
\sigma(s_1,\ldots, s_{n_3})$ in $\GG_1$ which connects
$\sigma_3$ and $\sigma_4$. We will consider several cases. If
$s_1 \in\{r_2,\ldots, j_1\}$ and $s_{n_3} \in\{ q_2,\ldots,
q_{n_2-1}\}$ then $\sigma_5$ connects $\sigma_1$ and $\sigma_2$
in $\GG_2$, a contradiction. We arrive at a contradiction for a
similar reason if we assume that $s_1 \in\{k_2,\ldots,
r_{n_1-1}\}$ and $s_{n_3} \in\{ q_2,\ldots, q_{n_2-1}\}$; or if
we assume that $s_1 \in\{q_2,\ldots, q_{n_2-1}\}$ and $s_{n_3}
\in\{ r_2,\ldots, j_1\}$; or $s_1 \in\{q_2,\ldots, q_{n_2-1}\}$
and $s_{n_3} \in\{ k_2,\ldots, r_{n_1-1}\}$. Next suppose that
$s_1 \in\{j_2,\ldots, k_1\}$ and $s_{n_3} \in\{ q_2,\ldots,
q_{n_2-1}\}$. Let $\sigma_6 = \sigma(j_1, \ldots, s_1)$ be the
sub-path of $\sigma_3$. Then the concatenation of $\sigma_6$ and
$\sigma_5$ connects $\sigma_1$ and $\sigma_2$ in $\GG_2$, a
contradiction once again. Finally, suppose that $s_1 \in\{q_2,\ldots,
q_{n_2-1}\}$ and $s_{n_3} \in\{ j_2,\ldots, k_1\}$. Let
$\sigma_7 = \sigma(s_{n_3}, \ldots, k_1)$ be the sub-path of
$\sigma_3$. Then the concatenation of $\sigma_5$ and $\sigma_7$
connects $\sigma_1$ and $\sigma_2$ in $\GG_2$, which is a
contradiction.

By the\vspace*{1pt} inductive assumption, $X_{r_1,\ldots, j_1, j_2,\ldots,
k_1, k_2,\ldots, r_{n_1}}$ and $X_{q_1,\ldots, q_{n_2}}$ are
conditionally independent given the values at the start and end
of the corresponding cell. This and the Markov property imply
that the process $\{X_{r_1,\ldots, j_1, j_2,\ldots, k_1, k_2,\ldots,
r_{n_1}}(t)$, $t\in[\V_j, \V_k]\}$ is conditionally
independent from the processes $\{X_{r_1,\ldots, j_1, j_2,\ldots, k_1,
k_2,\ldots, r_{n_1}}(t), t\in[\V_{r_1},
\V_{r_{n_1}}]\setminus[\V_j, \V_k]\}$ and\vspace*{1pt} $X_{q_1,\ldots,
q_{n_2}}$ given the values of $X_{r_1,\ldots, j_1, j_2,\ldots,
k_1, k_2,\ldots, r_{n_1}}(\V_j)$ and\vspace*{1pt} $X_{r_1,\ldots, j_1, j_2,\ldots,
k_1, k_2,\ldots, r_{n_1}}(\V_k)$. The claim remains valid
if we replace $\{X_{r_1,\ldots, j_1, j_2,\ldots, k_1, k_2,\ldots,
r_{n_1}}(t), t\in[\V_j, \V_k]\}$ with $\{X'_{jk} (t),
t\in[\V_j, \V_k]\}$, and this in turn shows that the joint
distribution of $X'_{r_1,\ldots, j_1, j, k,k_2,\ldots, r_{n_1}}$
and $X'_{q_1,\ldots, q_{n_2}}$ is the same as that of $X_{r_1,\ldots,
j_1, j_2,\ldots, k_1, k_2,\ldots, r_{n_1}}$ and $X_{q_1,\ldots,
q_{n_2}}$. Hence, $X'_{r_1,\ldots, j_1, j, k , k_2,\ldots, r_{n_1}}$
and $X'_{q_1,\ldots, q_{n_2}}$ are
conditionally independent given their values at the start and
end. We have shown that $X'$ is cell-Markovian.

Next we will show that $X'$ is time-Markovian. Suppose that $t
\notin E_{jk}$. If $F(t)$ and $P(t)$ in $\GG_2$ are the same as
$F(t)$ and $P(t)$ in $\GG_1$, then the time-Markov property
obviously holds for $t$ in $\GG_2$. Suppose that the future
$F_2(t)$ of $t$ in $\GG_2$ is the union of the future $F(t)$ of
$t$ in $\GG_1$ and $E_{jk}$. Since $\{X'(t), t\in E_{jk}\}$ is
the Markov bridge between $(t_j, X(t_j))$ and $(t_k, X(t_k))$
otherwise independent of $\{X'(t), t \in\EE_2 \setminus
\EE_1\}$, and, by the inductive assumption, $\{X(s), s\in
F(t)\}$ and $\{X(s), s\in P(t)\}$ are conditionally independent
given $X(t)$, it follows easily that $\{X'(s), s\in F_2(t)\}$
and $\{X'(s), s\in P(t)\}$ are conditionally independent given
$X'(t)$. A similar argument applies when the past $P_2(t)$ of
$t$ in $\GG_2$ is the union of the past $P(t)$ of $t$ in $\GG_1$
and~$E_{jk}$.

Consider the case when $t \in E_{jk}$. Let $\sigma$ be a full
path disjoint from $E_{jk}$ except for $t_j$ and $t_k$. By
Remark~\ref{m22.1}, the conditional distributions of $\{X(u),
t_j\leq u\leq t_k, u \in\sigma\}$, $\{X(u), u\in F(t_k)\}$ and
$\{X(u), u\in P(t_j)\}$ are independent given $X(t_j)$ and
$X(t_k)$. Moreover, the conditional distribution of $\{X(u),
t_j\leq u\leq t_k, u \in\sigma\}$ is that of a Markov bridge
between $(t_j, X(t_j))$ and $(t_k, X(t_k))$. This and the fact
that $\{X'(t), t\in E_{jk}\}$ is the Markov bridge between
$(t_j, X(t_j))$ and $(t_k, X(t_k))$ otherwise independent of
$\{X'(t), t \in\EE_2 \setminus\EE_1\}$ imply that the joint
distribution of $\{X(u), t_j\leq u\leq t_k, u \in E_{jk}\}$,
$\{X(u), u\in F(t_k)\}$ and $\{X(u), u\in P(t_j)\}$ is the same
as the joint distribution of $\{X(u), t_j\leq u\leq t_k, u \in
\sigma\}$, $\{X(u), u\in F(t_k)\}$ and $\{X(u), u\in P(t_j)\}$.
By the inductive assumption, the time-Markovian property holds
for $X$, $\GG_1$ and the point $t_*\in\sigma$ with the same
time coordinate as $t$, so we conclude that the time-Markovian
property holds for $X'$, $\GG_2$ and $t$. This completes the
proof of the time-Markovian property for $X'$.

We will now show that $X'$ is graph-Markovian. Suppose that
$\WW\subset\RR(\GG_2)$ is finite, and $\EE_0, \EE_1$ and $\EE_2$
are as in Definition~\ref{m23.1}. Let $\WW^* = \WW\cap
\RR(\GG_1)$, and, assuming that $\WW^* \ne\varnothing$, let
$\EE_0^*, \EE_1^*$ and $\EE_2^*$ be defined relative to $\GG_1$
as in Definition~\ref{m23.1}. By the induction assumption, the
conditional distribution of $\{X'(t), t\in E, E\in\EE_1^*\}$
given $\{X'(t), t\in E, E\in\EE_2^*\}$ depends only on
$\{X'(t), t\in\WW^*\}$. Since $\{X'(t), t\in E_{jk} \}$ is a
Markov bridge between $t_j$ and $t_k$, independent of the values
of $X'$ except for $X'(t_j)$ and $X'(t_k)$, it is easy to see
that the conditional distribution of $\{X(t), t\in E, E\in
\EE_1\}$ given $\{X(t), t\in E, E\in\EE_2\}$ depends only on
$\{X(t), t\in\WW\}$. If $\WW^* = \varnothing$, the same
conclusion is also evident. Hence, $\GG_2$ is graph-Markovian.

It remains to prove uniqueness in law of a natural $\PP$-process
on an NCC-graph. Once again, we use induction. The distribution
of a natural $\PP$-process on the ``minimal'' graph described
above is obviously unique. Suppose that we have shown uniqueness
in law for natural $\PP$-processes on all NCC-graphs with the number of
edges equal to $1+ 3r$ or less. Any NCC-graph $\GG_2=(\VV_2,
\EE_2)$ with $1+3(r+1)$ edges can be constructed from an
NCC-graph $\GG_1=(\VV_1, \EE_1)$ with $1+3r$ edges by adding an
edge, say $E_{jk}$, as in Definition~\ref{def:ind}(ii).
Consider a natural $\PP$-process $X'$ on $\GG_2$. Its
restriction $X$ to $\GG_1$ is a $\PP$-process. We will argue
that $X$ is a natural $\PP$-process on $\GG_1$.

First, we will prove that $X$ on $\GG_1$ is time-Markovian.
Consider any point $t$ in~$\GG_1$, and let $F_1(t)$ and $P_1(t)$
be the future and past of $t$ relative to $\GG_1$, defined as in
Definition~\ref{m23.2}. Let $F_2(t)$ and $P_2(t)$ be the future
and past of $t$ relative to $\GG_2$ and note that $F_1(t)
\subset F_2(t)$ and $P_1(t) \subset P_2(t)$. Since $X'$ on
$\GG_2$ is assumed to be natural, the conditional distributions
of $\{X'(s), s\in F_2(t)\}$ and $\{X'(s), s\in P_2(t)\}$ are
independent given $X'(t)$. This clearly implies that the
conditional distributions of $\{X(s), s\in F_1(t)\}$ and
$\{X(s), s\in P_1(t)\}$ are independent given $X(t)$. We see
that $X$ on $\GG_1$ is time-Markovian.

Next we will show that $X$ is graph-Markovian on $\GG_1$. Let
$\WW\subset\RR(\GG_1)$ be as in Definition~\ref{m23.1} and let
$\wh\EE_1$ and $\wh\EE_2$ play the roles of $\EE_1$ and
$\EE_2$ in the same definition (in this proof, $\EE_1$ and
$\EE_2$ denote the sets of edges of $\GG_1$ and $\GG_2$). Since
$X'$ is natural, the conditional distribution of $\{X'(t), t \in
E_{jk}\}$ given $\{X'(t), t\in E, E \in\EE_2 \setminus\{
E_{jk}\}\}$ is that of a Markov bridge between $(t_j, X'(t_j))$
and $(t_k, X'(t_k))$. For future reference, let us call this
property (A).

Suppose that $t_j\in E_1, t_k \in E_2$ for some $E_1,E_2\in\wh
\EE_1$, $t_j,t_k \notin\WW$ and let $\wt\EE_1 = \wh\EE_1 \cup
\{E_{jk}\}$. We have assumed that $X'$ is graph-Markovian, so the
conditional distribution of $\{X'(t), t\in E, E\in\wt\EE_1\}$
given $\{X'(t), t\in E, E\in\wh\EE_2\}$ depends only on
$\{X'(t), t\in\WW\}$. This and (A) easily imply that the
conditional distribution of $\{X(t), t\in E, E\in\wh\EE_1\}$
given $\{X(t), t\in E, E\in\wh\EE_2\}$ depends only on
$\{X(t), t\in\WW\}$.

The same argument applies when $t_j\in E_1, t_k \in E_2$ for
some $E_1,E_2\in\wh\EE_1$, $t_j \notin\WW$ and $t_k \in\WW$,
and also in the case when $t_j\in E_1, t_k \in E_2$ for some
$E_1,E_2\in\wh\EE_1$, $t_j \in\WW$ and $t_k \notin\WW$.

Consider the case when $t_j\in E_1, t_k \in E_2$ for some
$E_1,E_2\in\wh\EE_2$ and let $\wt\EE_2 = \wh\EE_2 \cup
\{E_{jk}\}$. We have assumed that $X'$ is graph-Markovian so the
conditional distribution of $\{X'(t), t\in E, E\in\wh\EE_1\}$
given $\{X'(t), t\in E, E\in\wt\EE_2\}$ depends only on
$\{X'(t), t\in\WW\}$. This and (A) easily imply that the
conditional distribution of $\{X(t), t\in E, E\in\wh\EE_1\}$
given $\{X(t), t\in E, E\in\wh\EE_2\}$ depends only on
$\{X(t), t\in\WW\}$.

Note that if $t\in\WW$, then $t\in E_1$ for some $E_1\in\wh
\EE_1$ and $t\in E_2$ for some $E_2\in\wh\EE_2$. Hence, the
only case that remains to be analyzed is when $t_j\in E_1, t_k
\in E_2$ for some $E_1 \in\wh\EE_1 ,E_2\in\wh\EE_2$,
$t_j,t_k \notin\WW$. Since $t_j\in E_1$ for some $E_1 \in\wh
\EE_1$, $t_j \notin\WW$,\vspace*{1pt} and taking into account how $E_{jk}$
was added to $\GG_1$, it follows that there exist $E_{t_1, t_j},
E_{t_j, t_2} \in\wh\EE_1$ such that $t_1, t_j, t_2, t_k$ lie
on a full time path $\sigma_1$. Since the process $X'$ is
natural, the conditional distributions of $\{X'(u), t_1\leq
u\leq t_2, u \in\sigma_1\}$, $\{X'(u), u\in F(t_2)\}$ and
$\{X'(u), u\in P(t_1)\}$ are independent given $X'(t_1)$ and
$X'(t_2)$, and, moreover, the conditional distribution of
$\{X'(u), t_1\leq u\leq t_2, u \in\sigma_1\}$ given $X'(t_1)$ and
$X'(t_2)$ is that of a Markov bridge between $(t_1, X'(t_1))$ and
$(t_2, X'(t_2))$. We will need the following two facts in the next step
of the argument. The first is property (A) defined above. The
second is an application of the graph-Markovian property for~$X'$.
Let $\EE_3$ be the union of all edges that comprises $\{u\dvtx
t_1\leq u\leq t_2, u \in\sigma_1\}$, $F(t_2)$, $P(t_1)$ and
$E_{jk}$. Let $\EE_4$ be the union of all edges such that the
union of $\EE_3$ and $\EE_4$ represents the whole graph $\GG_2$,
and $\WW_1$ is a finite set of points that $\EE_3$ and $\EE_4$
have in common. Note that $t_j, t_k \notin\WW_1$ because all
edges that end at these points belong to $\EE_3$. By the
graph-Markovian property of $X'$, the conditional distribution
of $\{X'(t), t\in E, E\in\EE_4\}$ given $\{X'(t), t\in E, E\in
\EE_3\}$ depends only on $\{X'(t), t\in\WW_1\}$.

Let $\DD_1$ be the distribution of $\{X'(t), t\in P(t_1) \cup
F(t_2)\}$. Let $\DD_2(x_1,x_2)$ be the conditional distribution
of $\{X'(u), t_1\leq u\leq t_2, u \in\sigma_1\}$ given
$\{X'(t_1) = x_1, X'(t_2) = x_2\}$. Let $\DD_3(x_j,x_k)$ be the
conditional distribution of $\{X'(u), u \in E_{jk}\}$ given
$\{X'(t_j) = x_j, X'(t_k) = x_k\}$. Let $\DD_4(\ol x)$ be the
conditional distribution of $\{X'(u), u \in E, E\in\EE_4\}$
given the sequence $\ol x$ of values of $X'$ at all points in
$\WW_1$.

We can construct a process $Y$ on $\GG_2$ with the same
distribution as $X'$ as follows. First, define a process $\{Y(t), t\in
P(t_1) \cup F(t_2)\}$ with
distribution $\DD_1$ on some probability space.
Then define a process $\{Y(u), u \in E, E\in\EE_4\}$ with
distribution $\DD_4(\ol y)$, independent of
$\{Y(t), t\in P(t_1) \cup F(t_2)\}$, except that $\ol y$ is the already
generated sequence of values of $Y$ on $\WW_1$.
Next define an independent (except for the
endpoints) Markov bridge $\{Y(u), t_1\leq u\leq t_2, u \in
\sigma_1\}$ between $(t_1, Y(t_1))$ and $(t_2, Y(t_2))$. This
process has distribution $\DD_2(Y(t_1), Y(t_2))$. Finally define an
independent (except for the endpoints) Markov bridge $\{Y(u), u
\in E_{jk}\}$ between $(t_j, Y(t_j))$ and $(t_k, Y(t_k))$. This
process has distribution $\DD_3(Y(t_j), Y(t_k))$. It follows from
our earlier remarks that $Y$ has the same distribution as $X'$
on $\GG_2$. The point of this construction is that it shows that
given $\{Y(t_1), Y(t_2)\}$, the distribution of $Y$ on $P(t_1)
\cup F(t_2) \cup\EE_4$ is independent of $\{Y(u), t_1\leq u\leq t_2, u
\in\sigma_1\}$. Hence, the
distribution of $X$ on $P(t_1) \cup F(t_2) \cup\EE_4$ is
independent of $\{X(u), t_1\leq u\leq t_2, u \in\sigma_1\}$ given $\{
X(t_1), X(t_2)\}$.

Let $\wt\EE_2 = \wh\EE_2 \cup\{E_{jk}\}$. Since\vspace*{1pt} $X'$ is
graph-Markovian, the conditional distribution of $\{X'(t), t\in
E, E\in\wt\EE_2\}$ given $\{X'(t), t\in E, E\in\wh\EE_1\}$
depends only on $\{X'(t), t\in\WW\cup\{t_j\}\}$. It follows
that the conditional distribution of $\{X(t), t\in E, E\in\wh
\EE_2\}$ given $\{X(t), t\in E, E\in\wh\EE_1\}$ depends only
on $\{X(t), t\in\WW\cup\{t_j\}\}$.
Note that the values of
$\{X(t), t\in E, E\in\wh\EE_1\}$ include the values of
$X(t_1)$ and $ X(t_2)$. Since the
distribution of $X$ on $P(t_1) \cup F(t_2) \cup\EE_4$ is
independent of $\{X(u), t_1\leq u\leq t_2, u \in\sigma_1\}$ given $\{
X(t_1), X(t_2)\}$, we conclude\vspace*{1pt} that the conditional
distribution of $\{X(t), t\in E, E\in\wh\EE_1\}$ given
$\{X(t), t\in E, E\in\wh\EE_2\}$ depends only on $\{X(t), t\in
\WW\}$. This completes the discussion of the last
remaining case of graph-Markovian property for $X$ on~$\GG_1$.

By assumption, the families of random variables $\{X'(t), t\in
E\}$, $E\in\EE_2$, are conditionally independent given
$\{X'(t), t\in\VV_2\}$, and for every $E_{jk}\in\EE_2$, the
conditional distribution of $\{X'(t), t\in E_{jk}\}$ given
$\{X'(t), t\in\VV\}$ is a Markov bridge between $(t_j,
X'(t_j))$ and $(t_k, X'(t_k))$. It is obvious that this implies
that the analogous property holds for $X$ on $\GG_1$. We have
already shown that $X$ is time-Markovian and graph-Markovian on
$\GG_1$, so $X$ is natural on $\GG_1$. By the induction
assumption, $X$ has a unique distribution. Rephrasing what we
said earlier in this paragraph, for $E_{jk}\in\EE_2\setminus
\EE_1$, the conditional distribution of $\{X'(t), t\in E_{jk}\}$
given $\{X'(t), t\in E, E\in\EE_1\}$ is a Markov bridge between
$(t_j, X'(t_j))$ and $(t_k, X'(t_k))$. This determines the
distribution of $X'$ uniquely.
\end{pf*}

Suppose that $\GG= (\VV,\EE)$ is an NCC TLG, and $\VV$ is
infinite. According to the definition of an NCC TLG with an
infinite vertex set, there exists a tower of NCC-graphs $\GG_n =
(\VV_n, \EE_n)$, $n\geq1$, such that each $\VV_n$ is finite,
$\VV= \bigcup_{n\geq1} \VV_n$ and $\EE= \bigcup_{n\geq1}
\EE_n$. Let $X_n$ be the natural $\PP$-process on $\GG_n$. By
Remark~\ref{m25.1}(iv), the restriction of $X_n$ to $\GG_k$,
for $k<n$, has the same distribution as that of $X_k$. A routine
application of Kolmogorov's consistency theorem shows that there
exists a $\PP$-process $X$ on $\GG$ such that its restriction to
any $\GG_k$ has the same distribution as that of $X_k$. Note
that the distribution of $X$ may depend, in principle, on the
sequence $\{\GG_n\}$. We will show that it does not if $\GG$ is
planar. We conjecture that the result holds for all NCC TLGs with
infinite $\VV$.
\begin{theorem}\label{m25.3}
Suppose that $\GG= (\VV,\EE)$ is a planar NCC TLG, and $\VV$ is
infinite. If $X$ and $X'$ are two $\PP$-processes on $\GG$
constructed using two towers of NCC-graphs $\{\GG_n\}_{n\geq1}$
and $\{\GG'_n\}_{n\geq1}$, then $X$ and $X'$ have the same
distributions.
\end{theorem}
\begin{pf}
Suppose that we can prove that for any $\GG_j=(\VV_j, \EE_j)$
and $\GG'_k= (\VV'_k, \EE'_k)$ there exist $n \geq\max(j,k)$
and graphs $\HH_m$, $m=j+1, \ldots, n$, and $\HH'_m$, $m=k+1,\ldots,
n'$ such that $\HH_n = \HH'_{n'}\subset\GG$ and $\GG_1,\ldots, \GG_j,
\HH_{j+1},\ldots, \HH_n$ and $\GG'_1,\ldots,
\GG'_k, \HH'_{k+1},\ldots, \HH'_{n'}$ are towers of NCC-graphs.
By Theorem~\ref{th:nat}(i) and its proof, we can construct a
natural process $Y$ on $\HH_n$ such that its restriction to
$\GG_j$ is~$X$, and we can construct a natural process $Y'$ on
$\HH'_{n'}$ such that its restriction to $\GG'_k$ is $X'$. By
the uniqueness in distribution of the natural process on an
NCC-graph, the distributions of $Y$ and $Y'$ are identical.
Hence, the distributions of $X$ and $X'$ agree on $\EE_j \cap
\EE'_k$. Letting $j,k \to\infty$, we conclude that the
distributions of $X$ and $X'$ agree on $\EE$.

It remains to prove that we can construct sequences $\{\HH_m\}$
and $\{\HH'_m\}$ with the properties listed above. First suppose
that the initial vertices $t_{0,n}\in\VV_n$ and $t'_{0,n}\in
\VV'_n$ and terminal vertices $t_{N_n,n}\in\VV_n$ and
$t'_{N_n,n}\in\VV'_n$ are the\vadjust{\goodbreak} same for all $\GG_n$ and
$\GG'_n$, that is, $t_0 := t_{0,n} = t'_{0,m} $ and $t_\infty:=
t_{N_n,n} = t'_{N_m,m}$ for all $n$ and $m$. Assume also that
the initial edges for both sequences overlap, that is, $E_{t_{0,n},
t_{1,n}} \subset E_{t'_{0,m}, t'_{1,m}}$ or $E_{t'_{0,m},
t'_{1,m}} \subset E_{t_{0,n}, t_{1,n}}$ for all $n$ and $m$.
Similarly, assume that terminal edges overlap, that is,
$E(t_{N_n-1,n}, t_{N_n,n}) \subset E(t'_{N_m-1,m}, t'_{N_m,m})$
or $E(t'_{N_m-1,m}, t'_{N_m,m}) \subset E(t_{N_n-1,n},
t_{N_n,n})$ for all $n$ and $m$. Moreover, we assume that $\GG_1
= \GG'_1$ is the same full time path.

Consider a (planar) representation $\RR(\GG)$ of $\GG$ and
suppose that $\RR(\GG_j) \subset\RR(\GG)$ and $\RR(\GG'_k)
\subset\RR(\GG)$. Recall that a representation of a planar
graph is a set of points $(t, x) \in\R^2$. It is easy to see
that the upper boundary of $\RR(\GG_j)$ is the graph of a
continuous function $f_j\dvtx[t_0, t_\infty] \to\R$, that is, $f_j(t)
= \max\{x\dvtx(t,x) \in\RR(\GG_j)\}$. We similarly define $f'_k$
relative to $\GG'_k$ and let $f = \max(f_j, f'_k)$. We then
define the lower boundary of $\RR(\GG_j)$ as the graph of a
continuous function $g_j\dvtx[t_0, t_\infty] \to\R$, that is, $g_j(t)
= \min\{x\dvtx(t,x) \in\RR(\GG_j)\}$, $g'_k$ as the lower boundary
of $\RR(\GG'_k)$ and $g = \min(g_j, g'_k)$. For any real
functions $a(t)$ and $b(t)$, let the graph $\KK_{a,b}$ be
defined by $\RR(\KK_{a,b}) = \{(t,x) \in\RR(\GG)\dvtx a(t) \leq x
\leq b(t)\}$. Let $\HH_n = \HH'_{n'} = \KK_{f,g} $.

Since $\GG_j$ is an element of an infinite tower of graphs
$\{\GG_m\}$ such that $\bigcup_{m\geq1} \RR(\GG_m) = \RR(\GG)$
and $\GG$ is locally finite, we must have $ \RR(\KK_{f_j, g_j})
\subset\RR(\GG_{m_1})$ for some $m_1 \geq j$. Let $E_1, E_2,\ldots,
E_{m_2}$ be edges added during the inductive construction
of the tower $\{\GG_m\}_{j \leq m \leq m_1}$ and such that their
representations are in $\RR(\GG_{m_1}) \setminus\RR(\KK_{f_j,
g_j})$, listed in the order in which they are added during the
inductive construction. We construct a tower $\GG_j, \HH_{j+1},\ldots,
\HH_{j+ m_2} = \KK_{f_j, g_j}$ by adding edges $E_1, E_2,\ldots,
E_{m_2}$ in the same order (and no other edges). This
construction can proceed according to the rules of the inductive
construction of NCC graphs because edges $E_1, E_2,\ldots,
E_{m_2}$ are shielded by the graphs of the functions $f_j$ and
$g_j$ from all other edges added during the construction of
$\{\GG_m\}_{j \leq m \leq m_1}$. We construct a tower $\GG'_k,
\HH'_{k+1},\ldots, \HH_{k+ m_3} = \KK_{f'_k, g'_k}$ in an
analogous way.\vspace*{2pt}

It remains to define $\HH_{j+ m_2+1},\ldots, \HH_n$. For future
reference, we label the next part of the proof ``Step (I).'' If
$f'_k(t) \leq f_j(t)$ and $g'_k(t) \geq g_j$
for all $t$ then we let $\HH_{j+ m_2+1} = \HH_{j+ m_2}$.
Otherwise, suppose without loss of generality that $f'_k(t) >
f_j(t)$ for some $t$. Let $t_1, t_2,\ldots, t_{m_4}$ be all vertices
in the graph of $f_j$ such that there is an edge in $\HH_n
\setminus\HH_{j+ m_2}$ ending in $t_r$, for $r=1,\ldots,
m_4$. The first such edge must go from $t_1$ forward in time, and
the last such edge must end in $t_{m_4}$. Hence, there must be a
pair of vertices $t_r$ and $t_{r+1}$ such that there is an edge
$E_r$ in $\HH_n \setminus\HH_{j+ m_2}$ starting from $t_r$ and
an edge $E_{r+1}$ (possibly the same edge) ending in $t_{r+1}$.
By the planarity of $\GG$, there must be a time path $\sigma$
from $t_r$ to $t_{r+1}$ in $\HH_n$ containing $E_r$ and
$E_{r+1}$. We add $\sigma$ (treated as a single edge) to
$\HH_{j+ m_2}$ and thus obtain $\HH_{j+ m_2+1}$.

If $f'_k(t) \geq f_j(t)$ and $g'_k(t) \leq g_j(t)$ for all $t$, then
we let $\HH'_{j+ m_3+1} = \HH'_{j+ m_3}$. Otherwise, we generate
$\HH'_{j+ m_3+1}$ in a way analogous to that used to construct
$\HH_{j+ m_2+1}$.

If $f'_k(t) = f_j(t)$ and $g'_k(t) = g_j(t)$ for all $t$, then we
let $\HH_n = \HH'_{n'} = \HH_{j+ m_2}$. In this case, we are
done. Otherwise, we have constructed towers of NCC-graphs
\[
\GG_1,\ldots, \GG_j,\qquad \HH_{j+1},\ldots, \HH_{j+
m_2+1}
\]
and
\[
\GG'_1,\ldots, \GG'_k,\qquad
\HH'_{k+1},\ldots, \HH'_{k+ m_3+1}
\]
such that either
$\HH_{j+ m_2+1}$ is strictly greater than $\GG_j$ or $\HH'_{k+
m_3+1}$ is strictly greater than $\GG'_k$, or both. Moreover,
the TLG analogous to $\KK_{f, g}$ but defined relative to
$\HH_{j+ m_2+1}$ and $\HH'_{k+ m_3+1}$ in place of $\GG_j$ and
$\GG'_k$ is the same as $\KK_{f, g}$.

We now proceed in an inductive way. Suppose that we constructed
$\HH_{r_1}$ and $\HH'_{r_2}$. Let $f_{r_1}$ represent the upper
boundary of $\HH_{r_1}$, let $f'_{r_2}$ represent the upper
boundary of $\HH'_{r_2}$, let $g_{r_1}$ represent the lower
boundary of $\HH_{r_1}$ and let $g'_{r_2}$ represent the lower
boundary of $\HH'_{r_2}$. We now repeat Step (I) with $f_j$
replaced by $f_{r_1}$, $f'_k$ replaced by $f'_{r_2}$, $g_k$
replaced by $g_{r_2}$ and $g'_k$ replaced by $g'_{r_2}$. This
will generate towers
\[
\GG_1,\ldots, \GG_j,\qquad
\HH_{j+1},\ldots, \HH_{r_1+1} \quad\mbox{and}\quad
\GG'_1,\ldots, \GG'_k, \qquad\HH'_{k+1},\ldots, \HH'_{r_2+1}.
\]
If $\HH_{r_1+1} = \HH'_{r_2+1} = \HH_n = \HH'_{n'}$,
then we are done. Otherwise $\HH_{r_1+1}$ is strictly greater
than $\HH_{r_1}$, or $\HH'_{r_2+1}$ is strictly greater than
$\HH'_{r_2}$, or both. The growth cannot continue forever
because $\KK_{f, g}$ has a finite number of edges, so eventually
we will have $\HH_{r_1+1} = \HH'_{r_2+1} = \HH_n = \HH'_{n'}$.

Next we will argue that one can drop the assumption that $\GG_1
= \GG'_1$ is the same full time path (but we keep the assumption
about overlapping of the initial edges and terminal edges of $\GG_j$'s and
$\GG'_k$'s). Suppose that $\RR(\GG_1) \cup
\RR(\GG'_1)$ contains only one cell. Then the cell has no edges inside.
Then the
argument given above will work under this weakened assumption
because $\RR(\HH_{j+m_2}) \cup\RR(\HH'_{k+m_3})$ will contain
all edges between the graphs of $f$ and~$g$.

A rather easy but tedious argument based on ideas used earlier
in this proof shows that for any two graphs (full time paths)
$\GG_1 $ and $ \GG'_1$ with initial edges and terminal edges
overlapping there exists a sequence of graphs $\JJ_1= \GG_1,
\JJ_2,\ldots, \JJ_q = \GG'_1$ such that $\RR(\JJ_r) \cup
\RR(\JJ_{r+1})$ contains only one cell, for
every $r$. This shows that $X$ and $X'$ have the same
distributions if $\GG_1 $ and $ \GG'_1$ have overlapping initial
and terminal edges.

Finally, we will show how to eliminate the assumption that the
initial and terminal edges of $\GG_1 $ and $ \GG'_1$ are
overlapping. Suppose that $\EE_* \subset\EE$ is a finite set,
and $\{\GG_n\}$ and $\{\GG'_n\}$ are two towers of NCC-graphs
increasing to $\GG$. Let $n_1$ be such that $\RR(\EE_*) \subset
\RR(\GG_{n_1}) \cap\RR(\GG'_{n_1})$. Let $t_*$ be a vertex such
that $t_* \in\VV$, $t_* \in\RR(\GG_1)$, and $t_*$ lies to the
left of all the vertices in $\VV_{n_1}$ and $\VV'_{n_1}$, except
the initial vertices. Note that there exists $n_2$ so large that
$t_* \in\VV'_{n_2} $. This shows that there is a time path
$\sigma'$ with the initial edge overlapping with the initial
edge of $\GG'_1$ and ending at $t_*$. Let $\sigma$ be the
initial part of $\GG_1$, between\vadjust{\goodbreak} $t_0$ and $t_*$. Let
$\GG''_{n_1}$ be the graph $\GG'_{n_1}$ with $\sigma'$ replaced
by~$\sigma$. Note that $\GG''_{n_1}$ is an NCC-graph because we
only changed the second coordinate of the representation of
$\GG'_{n_1}$ for a part of the graph. Let $X''$ be the natural
process on $\GG''_{n_1}$. The distribution of $X''$ on $\EE_*$
is the same as that of $X'$ because, once again, we only changed
the second coordinate of the representation of $\GG'_{n_1}$
for a part of the graph. We now modify the terminal part of
$\GG''_{n_1}$ to obtain an NCC-graph $\GG'''_{n_1}$ such that
the initial and terminal edges of $\GG'''_{n_1}$ and $\GG_{n_1}$
are overlapping. Let $X'''$ be the natural process on
$\GG'''_{n_1}$. The distribution of $X'''$ on $\EE_*$ is the
same as that of $X'$. And this is the same distribution as the
distribution of $X$ on $\EE_*$, by the first part of the proof.
Since $\EE_*$ is an arbitrary finite subset of $\EE$, we see
that $X$ and $X'$ have the same distributions.
\end{pf}
\begin{definition}
If $\GG= (\VV,\EE)$ is an NCC TLG with infinite $\VV$ and $X$
is a process on $\GG$ with the distribution as in Theorem
\ref{m25.3}, then we will call $X$ \textit{natural}.
\end{definition}

\section{Brownian motion on time-like graphs}\label{sec:BM}

In this section $\PP$ refers to the distribution of standard
Brownian motion. We will consider a TLG with a finite vertex set $\VV$,
$t_0 =0$ and $t_N=1$. The $\PP$-process $X$ on a TLG $\GG$ is a mean
zero Gaussian process so it is completely specified by its
covariance structure.
\begin{prop}\label{cov} If $(\sigma(j,j_{n_1},\ldots, j_{n_2},
k, j_{n_3},\ldots, j_{n_4}, n), \sigma(j,j_{n_5},\ldots,
j_{n_6}$, $m,j_{n_7},\ldots, j_{n_8},n))$ is a simple cell
of a TLG $\GG$, and $X$ is a natural Brownian motion on $\GG$,
then
\begin{eqnarray*}
&&\E( X_{j,j_{n_1},\ldots, j_{n_2}, k,
j_{n_3},\ldots, j_{n_4}, n}(\V_k)X_{j,j_{n_5},\ldots,
j_{n_6},m,j_{n_7},\ldots, j_{n_8},n}(\V_m) ) \\
&&\qquad= \V_j +
\frac{(\V_k-\V_j) (\V_m-\V_j)}{ (\V_n-\V_j)}.
\end{eqnarray*}
\upqed\end{prop}
\begin{pf}
We will abbreviate $X'= X_{j,j_{n_1},\ldots, j_{n_2}, k,
j_{n_3},\ldots, j_{n_4}, n}$ and
\[
X''= X_{j,j_{n_5},\ldots,
j_{n_6},m,j_{n_7},\ldots, j_{n_8},n}.
\]
By the cell-Markovian property
of the process, we can represent the
joint distribution of $X'$ and $X''$ as follows. Let $W$ and
$W'$ be independent Brownian bridges on the interval
$[\V_j,\V_n]$; in other words, $W$ and $W'$ are independent
Brownian motions conditioned by $W_{\V_j}=W'_{\V_j}= 0$ and
$W_{\V_n}= W'_{\V_n}= 0$. Let $\alpha= (\V_k-\V_j) /
(\V_n-\V_j)$ and $\beta= (\V_m-\V_j) / (\V_n-\V_j)$. Then the
conditional distribution of $X'(\V_k)$ given $\{X'(\V_j) = x_j,
X'(\V_n) = x_n \}$ is the same as the distribution of
$(1-\alpha) x_j + \alpha x_n + W_{\V_k}$. Moreover, the joint
distribution of $(X'(\V_k), X''(\V_m))$ given $\{X'(\V_j) = x_j,
X'(\V_n) = x_n \}$ is the same as the distribution of
$((1-\alpha) x_j + \alpha x_n + W_{\V_k}, (1-\beta) x_j + \beta
x_n + W'_{\V_m})$. Therefore,
\begin{eqnarray*}
&&\E\bigl(
X'(\V_k)X''(\V_m) \mid X'(\V_j) = x_j, X''(\V_n) = x_n\bigr)\\
&&\qquad= \E\bigl( \bigl((1-\alpha) x_j + \alpha x_n + W_{\V_k}\bigr)\bigl( (1-\beta)
x_j + \beta x_n + W'_{\V_m}\bigr)\bigr)\\
&&\qquad= \bigl((1-\alpha) x_j +\alpha
x_n \bigr)\bigl( (1-\beta) x_j + \beta x_n\bigr)
\end{eqnarray*}
and
\begin{eqnarray*}
&&\E( X'(\V_k)X''(\V_m) ) \\
&&\qquad= \E
\bigl(\bigl((1-\alpha) X'(\V_j) +\alpha X'(\V_n) \bigr) \bigl( (1-\beta)
X'(\V_j) + \beta X'(\V_n)\bigr)\bigr)\\
&&\qquad= \bigl((1-\alpha)(1-\beta) +
(1-\alpha)\beta+ \alpha(1-\beta)\bigr) \V_j + \alpha\beta\V_n\\
&&\qquad=
(1-\alpha\beta) \V_j + \alpha\beta\V_n = \V_j + \alpha\beta
(\V_n - \V_j)\\
&&\qquad= \V_j + \frac{(\V_k-\V_j) (\V_m-\V_j)}{
(\V_n-\V_j)^2} (\V_n - \V_j) \\
&&\qquad= \V_j + \frac{(\V_k-\V_j)
(\V_m-\V_j)}{ (\V_n-\V_j)}.
\end{eqnarray*}
\upqed\end{pf}
\begin{pf*}{Proof of Theorem~\ref{th:nat}\textup{(ii)}}
Suppose that for a TLG $\GG$, there exist simple coterminal
cells $(\sigma_1,\sigma_2)$ with endpoints $t_1<t_2$, and
$(\sigma_3,\sigma_4)$ with endpoints \mbox{$t_3<t_4$}. Moreover, either
$t_1 < t_3$ or $t_2 < t_4$. Assume that there exists a natural
Brownian motion $X$ on $\GG$. We will show that this assumption
leads to a contradiction.

We will assume without loss of generality that $\V_{2}<\V_{4}$
and $\V_1=\V_3$. The first edge of $\sigma_1$, say, $E_{1k}$,
must be the same as the first edge of $\sigma_3$ or the first
edge of~$\sigma_4$. Suppose without loss of generality that the
first edge of $\sigma_1$ is the same as the first edge of
$\sigma_3$. Then the first edge of $\sigma_2$, say $E_{1m}$, is
the same as the first edge of~$\sigma_4$. Then we can use
Proposition~\ref{cov} to express the covariance of $X$ at
vertices $\V_k$ and~$\V_m$. If we use the formula relative to
the cell $(\sigma_1,\sigma_2)$, then the answer is
\[
\V_1 + \frac{(\V_k-\V_1) (\V_m-\V_1)}{ (\V_{2}-\V_1)}.
\]
If we apply the same proposition relative to the
cell $(\sigma_3,\sigma_4)$, then we obtain a different answer,
\[
\V_1 + \frac{(\V_k-\V_1) (\V_m-\V_1)}{
(\V_{4}-\V_1)}.
\]
This contradiction shows that there
is no natural Brownian motion on $\GG$.
\end{pf*}

\section{Graph martingales and Harnesses}\label{sec:harn}

The previous computation of Brownian covariance in Section
\ref{sec:BM} can be extended to a class of processes called
harnesses. This class of processes, which includes all
integrable L\'evy processes and their bridges, was introduced
originally by Hammersley \cite{H67}. We follow the definition
given in the article by Mansuy and Yor \cite{MY05}.
\begin{definition} Suppose that $\TT\subset\R$ is a
bounded or unbounded
interval, and let $\{ H(t) , t \in\TT\}$ be an
integrable process for all $t$ whose sample paths are RCLL (right
continuous with left limits)
almost surely.\vadjust{\goodbreak}
Consider a past-future
filtration $(\ppf_{t,T}, t < T; t,T\in\TT)$, with the
property that
\[
\sigma\{ H(s); s \le t \mbox{ and }
s \ge T \} \subset\ppf_{t,T} \quad\mbox{and}\quad
\ppf_{t_1, T_1} \subseteq\ppf_{t, T},\qquad t_1 \le t < T \le T_1.
\]
The process $H$ is said to be a \textit{harness} with respect to
the filtration $(\ppf_{t,T}, t < T; t,T\in\TT)$ if, for all
$a < b < c < d$, we have
%
%
\begin{equation}\label{eq:hness} \E\biggl( \frac{H(c)
- H(b)}{c - b} \biggm|\ppf_{a,d} \biggr) = \frac{H(d) -
H(a)}{d-a}.
\end{equation}
\end{definition}

The equality in (\ref{eq:hness}) may also be reformulated as:
$H$ is a harness if and only if for all $s < t < u$, we get
%
%
\begin{equation}\label{eq:hness2} \E( H(t) \mid\ppf_{s,u}
) =
\frac{t-s}{u-s} H(u) + \frac{u-t}{u-s} H(s).
\end{equation}
The following lemma
establishes more path properties.
\begin{lemma}\label{lem:propharness} Let $\{ Y(t), t \in
[0,1] \}$ be a harness with respect to some past-future
filtration $(\ppf_{t,T}, t < T; t,T\in[0,1])$. Then the
following properties hold:
\begin{longlist}
\item The set of
random variables $\{ Y(s), 0\le s \le1\}$ is uniformly
integrable;
\item $Y$ is continuous in probability, that is,
for any $0 \le t \le1$ we have
%
%
\begin{equation}\label{eq:cont} P\Bigl(
\lim_{s\rightarrow t} Y(s)=Y(t) \Bigr)=1.
\end{equation}
\end{longlist}
\end{lemma}
\begin{pf} To prove (i), consider the collection of random
variables $\{ Y(s),\break 0\le s\le1/2 \}$. By (\ref{eq:hness2}),
for $t=1/2$, $u=1$, we get
\[
\frac{1-2s}{2(1-s)} Y(1) +
\frac{1}{2(1-s)} Y(s) = \E\bigl( Y(1/2) \mid\ppf_{s,1} \bigr).
\]
As $s$ varies between $0$ and $1/2$, the collection of
conditional expectations on the right is clearly uniformly
integrable. Thus, by rearranging terms and noting that $Y_1$ is
integrable, we get $\{ Y(s), 0\le s \le1/2 \}$ is also
uniformly integrable. By a similar argument one gets uniform
integrability of $\{ Y(s), 1/2 \le s \le1\}$, and this shows
uniform integrability of the entire process.

For (ii), recall that we consider only right continuous harnesses.
So it remains to
prove that $Y$ is continuous in probability from the
left at time $t$. By applying (\ref{eq:hness2}), for any $s < u < t <
T$ we get
%
%
\begin{equation}\label{nov29.1}
\E( Y(u) \mid\ppf_{s,T}
) = \frac{u-s}{T-s} Y(T) + \frac{T-u}{T-s} Y(s).
\end{equation}
Now we take $u$ approaching $t$ from the left. By uniform integrability
we get
\begin{eqnarray*}
\lim_{u\uparrow t} \E( Y(u) \mid\ppf_{s,T})&=& \E( Y(t-) \mid\ppf_{s,T})
= \frac{t-s}{T-s} Y(T) + \frac{T-t}{T-s} Y(s)\\
&=& \E( Y(t)
\mid\ppf_{s,T}).
\end{eqnarray*}
In other words, for all $s < t < T$ we get $\E( Y(t-) \mid\ppf_{s,T}
)=\E( Y(t) \mid\ppf_{s,T}
)$.
Now we take $T \downarrow t$ and use martingale convergence theorem
(see \cite{KS}, page~18) to claim $\E( Y(t-) \mid\ppf
_{s,t+}
)=\E( Y(t) \mid\ppf_{s,t+} )$, where $\ppf_{s, t+}=
\bigcap_{T
> t}\ppf_{s,T}$. Finally we take $s \uparrow t$ and the martingale
convergence theorem to claim
\[
\E( Y(t-) \mid\ppf_{t-,t+})=\E( Y(t) \mid\ppf_{t-,t+}
).
\]
Here $\ppf_{t-,t+} = \bigvee_{s < t} \ppf_{s, t+}$. But $Y(t-)$ is
obviously measurable with respect to $\ppf_{t-, t+}$ and $Y(t+)$ is
also measurable due to assumed right continuity. Hence $Y(t-)=Y(t)$
almost surely.
This shows (\ref{eq:cont}) and completes the proof of the lemma.
\end{pf}

To discuss the properties of a harness on a TLG $\GG$ we need to
introduce a few definitions. Recall that a path in a graph is
any sequence of vertices $(k_1, \ldots, k_m)$ such that adjacent
vertices $(k_j, k_{j+1})$ have edges in the graph $\GG$. The
only difference between a path and a time-path is that we do not
require the vertices to be increasing.
\begin{definition}\label{def:spine}
Let $\GG$ be an NCC TLG with finite $\VV$.
Consider a full time path $\sigma^*=\sigma(k_1, \ldots, k_n)$
and a point $t^*\in E_{j^*k^*}$. Consider the subgraph
$\GG^*=(\VV^*, \EE^*)$ where $\VV^*$ consists of all vertices $v
\in\GG$ such that there exists a path starting at $v$ and
ending at $j^*$ or $k^*$, and the path does not include any
vertex in $\sigma^*$. The edges of this subgraph are the edges
in $\GG$ such that both its vertices are included in $\VV^*$.
The full time path $\sigma^*$ will be called a \textit{support}
for $t^*$ if the subgraph $\GG^*$ is a tree. In other words, if
we remove the time path $\sigma^*$ from the graph $\GG$, then
the connected component of the remaining subgraph that contains
$t^*$ is a tree.
\end{definition}

Let $\PP$ denote the law of a Markovian harness in $[0,1]$.
Consider a natural $\PP$-process on an NCC TLG $\GG$. Suppose a full
time path $\sigma^*=\sigma(k_1, \ldots, k_n)$ is a support for a
nonvertex point $t^*$
on an edge $E_{j^*k^*}$. We want to know what $E( X_{j*k*}(t^*)
\mid X(t), t \in\sigma^*)$ is. The answer will be expressed
using a filtration constructed as follows.

Let $\WW_1$ denote the two vertices $\{ t_{j^*}, t_{k^*}\}$. Let
$\GG_1$ denote the subgraph of $\GG$ with the edge $E_{j^*k^*}$
removed. Or, equivalently, in any representation of $\GG$, we
remove the interior of the set $E_{j^*k^*}$. Let $\HH_1$ denote
the $\sigma$-algebra generated by the set of all random
variables $\{X_E(t), t\in E, E \in\GG_1\}$. Note that the
vertices $j^*$ and $k^*$ have degree two in the graph $\GG_1$
since the common edge gets deleted.

Now we proceed by induction. Suppose we have constructed $\WW_m,
\GG_m$, and $\HH_m$ such that every $t_i\in\WW_m$ has degree
two in the graph $\GG_m$. To construct $\WW_{m+1}$, consider
sequentially every vertex $t_i \in\WW_m$. If $t_i$ is a vertex
in $\sigma^*$ (i.e., $t_i$ is one of $\{t_{k_1}, \ldots,
t_{k_n}\}$), then $t_i$ continues to be in $\WW_{m+1}$. This, in
particular, holds true if $t_i$ is $0$ or $1$ which are in
$\sigma^*$. In this case we define the set of
\textit{descendants of $t_i$}, $\NN(t_i)$, as the singleton set\vadjust{\goodbreak}
$\{t_i\}$. Otherwise, $\NN(t_i)$ consists of the two distinct
neighbors of $t_i$ in the graph $\GG_m$. We define the set
$\WW_{m+1}$ as
\[
\WW_{m+1} = \bigcup_{t_i \in\WW_m} \NN(t_i).
\]
The subgraph $\GG_{m+1}$ is obtained from $\GG_m$ by deleting
all the vertices of $\WW_m$ not included in $\WW_{m+1}$ and all
their incident edges. The $\sigma$-algebra $\HH_{m+1}$ is
defined to be the one generated by all the random variables $\{
X_E(t), t \in E, E \in\GG_{m+1} \}$.

We stop the inductive process at the first $K$ when all vertices
in $\WW_K$ are in $\sigma^*$, which gives us a backward
filtration
\[
\HH_K \subset\HH_{K-1} \subset\cdots\subset
\HH_1.
\]

\begin{lemma}\label{a8.1}
Suppose that $\sigma^*$ is a support of $t^*$ so, by definition, $\GG
^*$ is a tree.

\begin{longlist}
\item Unless $t_i$ is in $\sigma^*$, it
cannot have a descendant already present in~$\WW_m$.

\item Any $v \in\WW_{m+1}$ which is not
included in $\sigma^*$ has exactly two neighbors in the graph
$\GG_{m+1}$.

\item There exists a tower of NCC graphs
$(\GG'_1, \GG'_2,\ldots, \GG'_M, \GG_K,\ldots, \GG_{K-1},\ldots,$ $\GG
_{K-2},\ldots, \GG_1, \GG)$, where $\GG'_1$ is a graph with a single
time path. In other words, every graph $\GG_m$, $m = 1,\ldots, K$, is
an element of this tower of NCC graphs. It is not necessarily true that
$\GG_m$'s are consecutive elements in this tower.
\end{longlist}
\end{lemma}
\begin{pf}
To see (i), consider two vertices $t_1 < t_2$ in $\WW_{m}$.
Note that there is always a path of the form $(t_1, u_1, \ldots,
u_k, t_2)$ such that $\{u_1, \ldots, u_k\}$ is in
$\bigcup_{i=1}^{m-1}\WW_i$. If $t_1$ and $t_2$ are neighbors,
then that creates a loop in the graph $\GG^*$ in Definition
\ref{def:spine}. Since we have assumed the graph $\GG^*$ to be a
tree, this is impossible.

For (ii), note that, any $v\in\WW_{m+1}$ which is not in
$\sigma^*$ is a neighbor to some distinct vertex in $\WW_m$ and
that edge has been deleted in $\GG_{m+1}$. Since the degree of
every nonterminal vertex is three, it remains to show that $v$
cannot be a neighbor to two (or three) vertices in $\WW_m$.

Assume on the contrary that there is a vertex $v\in\WW_{m+1}$
which is a neighbor of both $u_1 < u_2$, where $u_1, u_2 \in
\WW_m$. Since $v \notin\sigma^*$, this produces another loop in
the graph $\GG^*$ the possibility of which has been ruled out by
our assumption.

(iii)
Let $A$ be the connected component of $\RR(\GG) \setminus\sigma^*$
that contains $t^*$. Then
$\RR(\GG) \setminus A = \RR(\GG_K)$, by construction. We can reverse
the construction presented before the lemma based on deleting edges. In
the reversed construction we add edges one at a time, not in batches,
to obtain a tower of graphs
$(\GG_K,\ldots, \GG_{K-1},\ldots, \GG_{K-2},\ldots, \GG_1, \GG)$.
Every graph $\GG_m$, $m = 1,\ldots, K$, is an element of this tower of
graphs, but $\GG_m$'s are not necessarily consecutive elements.\vadjust{\goodbreak}

We will argue that $\GG_K$ is an NCC-graph. Suppose that $\GG_K$ is not
an NCC graph. Then, according to Definition~\ref{def:nci}(v), there
are minimal co-terminal cells
$(\sigma_1, \sigma_2)$ and $(\sigma_3,\sigma_4)$ in $\GG_K$. Since $A$
is a connected component of \mbox{$\RR(\GG) \setminus\sigma^*$}, it is easy
to see that both cells
$(\sigma_1, \sigma_2)$ and $(\sigma_3,\sigma_4)$ will stay minimal if
we add $A$ to $\GG_K$. Hence, these cells will be minimal co-terminal
cells in $\GG$. This contradicts the assumption that $\GG$ is NCC and
finishes the proof that $\GG_K$ is NCC. Hence, there exists a tower of
NCC graphs
$(\GG'_1, \GG'_2,\ldots, \GG'_M, \GG_K)$, where $\GG'_1$ contains only
one full time path. We can concatenate this tower and $(\GG_K,\ldots,
\GG_{K-1},\ldots, \GG_{K-2},\ldots, \GG_1, \GG)$ to obtain a single
tower of NCC graphs
$(\GG'_1, \GG'_2,\ldots, \GG'_M, \GG_K,\ldots, \GG_{K-1},\ldots, \GG
_{K-2},\ldots, \GG_1, \GG)$.
\end{pf}
\begin{prop}\label{prop:mgle} Let $\GG$ be a TLG with a full
time path $\sigma^*$ that is a support for a time point $t \in
E_{j^*k^*}$. Let $X$ be a natural $\PP$-Markovian harness on
$\GG$.

Let $\{\beta(u), u \ge0\}$ be a one-dimensional Brownian
motion independent of the $\PP$-harness $X$, with $\beta(0) = t$. We
define the
sequence of stopping times
\[
\sigma_1 = \inf\bigl\{ u \ge
0\dvtx\beta(u) \in\{ t_{j*}, t_{k*} \} \bigr\},
\]
and then
inductively,
\[
\sigma_{m+1} = \inf\{ u\ge\sigma_m\dvtx
\beta(u) \in\NN( \beta({\sigma_m}) ) \}.
\]

Then, for any $m=1,2,\ldots, K$, we get
\[
\E(
X_{j^*k^*}(t) \mid\HH_m )= \E_{\beta}[ X(
\beta({\sigma_m}) )].
\]
Here $\E_{\beta}$ is the
expectation with respect to the law of $\beta$, when the values
of the process $X$ are given.
\end{prop}
\begin{pf} Consider the case of $m=1$. By the graph-Markovian
property of the process $X$, it is clear that
\[
\E(
X_{j^*k^*}(t) \mid\HH_1 )= \E( X_{j^* k^*} \mid
X(t_{j^*}), X(t_{k^*}) ).
\]
Now, applying the harness
property (\ref{eq:hness2}), we get
\begin{eqnarray*}
\E(
X_{j^*k^*}(t) \mid\HH_1 )&=&
\frac{t-t_{j^*}}{t_{k^*}-t_{j^*}} X(t_{k^*}) +
\frac{t_{k^*}-t}{t_{k^*}-t_{j^*}} X(t_{j^*})\\
&=&\PPP\bigl(
\beta( \sigma_1 )= t_{k^*} \bigr)X(t_{k^*}) +
\PPP\bigl( \beta( \sigma_1 )= t_{j^*}
\bigr)X(t_{j^*})\\
&=& \E_{\beta}[ X( \beta(\sigma_1)
)].
\end{eqnarray*}

We now proceed by induction. Suppose that
%
%
\begin{equation}\label{eq:indmgle}
\E( X_{j^*k^*}(t) \mid\HH_m )= \sum_{t_i \in
\WW_{m}} \PPP\bigl( \beta( \sigma_m)= t_i \bigr)
X(t_i).
\end{equation}
Then, by the the tower property of conditional
expectations, we get
%
%
\begin{equation}\label{eq:tower} \E( X_{j^*k^*}(t)
\mid\HH_{m+1} )= \sum_{t_i \in\WW_{m}} \PPP\bigl(
\beta( \sigma_m)= t_i \bigr) \E( X(t_i) \mid
\HH_{m+1}).
\end{equation}

Now there are two cases to consider. First suppose that $t_i$ is
in the fixed full time path $\sigma^*$, in which case it is
measurable with respect to $\HH_{m+1}$, and thus $\E(
X(t_i) \mid\HH_{m+1})= X(t_i)$.

The other case is when $t_i \notin\sigma^*$. Note that, since
the degree of the vertex $t_i$ is exactly two in the graph
$\GG_m$, there are two vertices $v_1$ and $v_2$ such that if we
remove these two vertices, $t_i$ is disconnected from the rest
of the graph.
By Lemma~\ref{a8.1}(iii) and Remark~\ref{m25.1}(iv), the restriction
of $X$ to $\GG_m$ is a natural $\PP$-process.
Thus, from the graph-Markovian property of $X$ on $\GG_m$ and harness
property
\begin{eqnarray*}
\E( X(t_i) \mid\HH_{m+1} ) &=& \E(
X(t_{i})\mid X_{v_1}, X_{v_2} )\\
&=& \PPP\bigl( \beta(
\sigma_{m+1} )= t_{v_1} \mid\beta(\sigma_m)=t_i \bigr)
X(t_{v_1}) \\
&&{}+ \PPP\bigl( \beta( \sigma_{m+1} )= t_{v_2}
\mid\beta(\sigma_m)=t_i \bigr)X(t_{v_2}).
\end{eqnarray*}
Substituting this expression back in (\ref{eq:tower}) and
(\ref{eq:indmgle}) we get
\[
\E( X_{j^*k^*}(t) \mid
\HH_{m+1} )= \sum_{t_i \in\WW_{m+1}} \PPP\bigl( \beta
(
\sigma_{m+1})= t_i \bigr) X(t_i)=
\E_{\beta}[X(\beta(\sigma_{m+1}))].
\]
This completes
the proof of the proposition.
\end{pf}
\begin{theorem}\label{prop:skorokhod} Let $\mu$ be any
probability distribution on $[0,1]$. Let $Y$ be a Markovian
harness with law $\PP$.

\mbox{}\hphantom{i}\textup{(i)} Given any $\varepsilon>0$ and any metric $\rho$ which induces
the topology of weak convergence, it is possible construct a NCC
TLG $\GG$, a time point $t^* \in E_{j^*k^*}$, and a full time
path $\sigma^*=\sigma(t_{k_1}, t_{k_2}, \ldots, t_{k_n})$ such
that for a natural harness $X$ on $\GG$ with law~$\PP$, the
difference between the laws of the random variables
\[
E(
X_{j^*k^*}(t^*) \mid X_E(s), 0\le s \le1, E\in\sigma^*
) \quad\mbox{and}\quad \int_0^1 Y(s) \mu(ds)
\]
is less
than $\varepsilon$ in the metric $\rho$.

\textup{(ii)} If $\PP$ is the Wiener measure on $[0,1]$, it follows that
for any time point $u\in\sigma^*$, one can make the difference
between
%
%
\begin{equation}
\label{eq:427} E( X_{j*k*}(t^*) X_{\sigma^*}(u) )
\quad\mbox{and}\quad \int_0^u s \mu(ds) + u\mu(u,1]
\end{equation}
smaller than
$\varepsilon$.
\end{theorem}
\begin{pf} We use Dubins's solution to the Skorokhod embedding
problem. Please see the original article by Dubins \cite{D68}
for more details, or page 332 in the survey article by Ob\l\'oj
\cite{O04} (which treats the case when $\mu$ is continuous).\vadjust{\goodbreak}

Given a measure $\mu$ with support in $[0,1]$, the Skorokhod
problem asks for a stopping time $\tau$ with respect to the
Brownian filtration such that a standard one-dimensional
Brownian motion $\beta$ stopped at $\tau$ has law $\mu$. The
following is a solution proposed by Lester Dubins.

Consider any probability measure $\nu$ supported on $[0,1]$. For
any finite sequence $s$ of $0$'s and $1$'s starting with $0$, we
will define a probability measure $\nu_s$. Let \mbox{$\nu_{(0)}=\nu$}.
Suppose that $s_1 = (s, 0)$ and $s_2 = (s, 1)$ (this notation is
not quite rigorous but it is quite clear). It will suffice to
define $\nu_{s_1}$ and $\nu_{s_2}$ as functions of $\nu_s$. If
$\nu_s$ is supported on exactly one point then we let $\nu_{s_1}
= \nu_{s_2} =\nu_s$. Otherwise we consider $\nu_s$ restricted to
intervals $[0, \E(\nu_s))$ and $[\E(\nu_s), 1]$. We renormalize
both measures and thus we obtain $\nu_{s_1}$ and $\nu_{s_2}$.

Let $H_n(\mu)$ be the set of all numbers $\E(\nu_s)$ for all
sequences $s$ of length $n+1$, $n\geq0$, where $\nu_{(0)}=\mu$.
The sequence $\{ H_n(\mu), n=0,1,2,\ldots\}$ can be
naturally represented as a tree where every vertex has two
descendants unless it is a vertex that is repeated forever.

Let $\beta$ denote a one-dimensional Brownian motion such that
$\beta(0)=\E(\mu)$. Define $\tau_0\equiv0$ and define the
successive stopping times
\[
\tau_{n+1} = \inf\{ t\ge
\tau_n\dvtx\beta(t) \in H_{n+1}(\mu) \},\qquad
n=0,1,2,\ldots.
\]
Then Dubins shows that the distribution
$\mu_n$ of $\beta(\tau_n)$ is supported on at most $2^n$ many
atoms, and moreover $\mu_n$ converges to $\mu$ weakly as $n$
tends to infinity.

We will later show that
%
%
\begin{equation}\label{eq:n29.2}
\E_{\beta}( Y(\beta(\tau_n)) ) \to\int Y(s) \,d\mu(s)
\end{equation}
weakly as $n$ tends to infinity.

Assuming that (\ref{eq:n29.2}) is true, for any $\varepsilon>0$,
there exist a large enough $N$ such that the $\rho$-distance between
$\int Y(s) \,d\mu_n(s)$ and $\int Y(s) \,d\mu(s)$ is smaller than
$\varepsilon$. This is enough to prove part (i) of the proposition
since we can construct a tree $\tree_N$ with vertices
$\bigcup_{i=0}^N H_i(\mu)$ with an obvious tree structure. We
add to $\tree_N$ a full time path $\sigma^*$ by connecting
$\{0,1\}$ with all the elements in $H_N$. Finally we delete the
vertex at $\E(\mu)$ and name this time point $t^*$. Then
$\sigma^*$ is a support for the point $t^*$. This and
Proposition~\ref{prop:mgle} imply part (i) of the proposition.

For part (ii) we note that when $Y$ is Brownian motion, the weak convergence
(\ref{eq:n29.2}) entails convergence in $\mathbb{L}^2$. This is a
standard result
for linear combination of Gaussian processes that follows by
considering pointwise convergence of
the characteristic function. In other words, one can construct a tree
as above with an $N$
large enough such that the $\mathbb{L}^2$ distance between
\[
\E( X_{j*k*}(t) \mid\mathcal{H}_N ) \quad\mbox
{and}\quad
\int X_{\sigma*}(s) \,d\mu(s)
\]
is appropriately small. Now part (ii) follows by applying the
Cauchy--Schwarz inequality since the right-hand side of
(\ref{eq:427}) is the covariance between
\[
\int X_{\sigma*}(s) \,d\mu(s)\quad\mbox{and}\quad X_{\sigma*}(u).
\]

We return to the proof of (\ref{eq:n29.2}). It follows from
Dubins's construction that there is a limiting stopping time
$\tau$ such that $\lim_{n\rightarrow\infty} \tau_n = \tau$
almost surely, and $\lim_{n\rightarrow\infty}
\beta(\tau_n)=\beta(\tau)$, where $\beta(\tau)$ has law $\mu$.
Since $\beta$ is independent of $Y$ which is continuous in
probability (Lemma~\ref{lem:propharness}), it follows that
\[
\lim_{n\rightarrow\infty} Y(\beta(\tau_n))= Y(\beta(\tau))\qquad
\mbox{with probability one}.
\]
By Lemma~\ref{lem:propharness}
we know that $\{ Y(s), 0\le s \le1\}$ is uniformly
integrable, so the above shows that
\[
\lim_{n\rightarrow
\infty} \E_{\beta}( Y(\beta(\tau_n)) )= \E_{\beta}
Y(\beta({\tau}))\qquad \mbox{with probability one}.
\]
This completes
the proof of the theorem.
\end{pf}

\section{Brownian motion on honeycomb graph}\label{sec:hcomb}

We will prove a limit theorem for natural Brownian motion on the
honeycomb graph, when the diameter of hexagonal cells goes to
zero.
We will use the term ``Brownian motion'' to denote the two-sided
Brownian motion on the real line conditioned to be equal to 0 at
time~0.

Let $\RR(\GG^*_\rho)$ consist of the boundary of a single
hexagon with diameter $\rho>0$, with two of its sides parallel
to the first axis, and the leftmost vertex at $(0,0)$. Let
$\RR(\GG_\rho)$ be the usual hexagonal lattice in the whole
plane, containing $\RR(\GG^*_\rho)$ as a subset. It is easy to
see that there exists a tower $\{\GG^n_\rho\}$ of NCC TLGs with the
limit $\GG_\rho$, satisfying Definition~\ref{m27.1}(ii) so
$\GG_\rho$ is NCC. Hence, there exists a natural Brownian motion
$X$ on $\GG_\rho$.

Recall from Theorem~\ref{m25.3} that the distribution of $X$
does not depend on the tower of NCC-graphs used in the inductive
construction. Images of elements of a tower of NCC-graphs under
the symmetry with respect to the horizontal axis form another
tower converging to $\GG_\rho$. The same can be said about
images under vertical shifts by the hexagonal cell height. This
implies that the natural Brownian motion $X$ on $\GG_\rho$ is
invariant under the symmetry with respect to the horizontal axis
and under vertical shifts.
\begin{theorem}\label{p:n14.1} Let $X$ be the natural Brownian
motion on $\GG_\rho$. Consider $(u,0), (v, x) \in\R^2$ with
$u,v,x>0$. Let $(u_\rho, 0_\rho)$ be one of the vertices in
$\RR(\GG_\rho)$ with the smallest distance to $(u,0)$, and let
$E^\rho_u$ be an edge of $\GG_\rho$ that contains this vertex.
Let $(v_\rho, x_\rho)$ be one of the vertices in $\RR(\GG_\rho)$
with the smallest distance to $(v,4 x/(\sqrt{3}\rho))$, and let
$E^\rho_v$ be an edge of $\GG_\rho$ that contains this vertex.
Let $\Phi$ be the standard normal\vadjust{\goodbreak} cumulative distribution
function, that is, $\Phi(a) = (1/\sqrt{2\pi}) \int_{-\infty}^a
e^{-s^2/2} \,ds$. Then
\begin{eqnarray*}
&&\lim_{\rho\to0}
\E(X_{E^\rho_u}(u_\rho) X_{E^\rho_v}(v_\rho)) \\
&&\qquad= \frac
{\sqrt{5x}} {8\sqrt{\pi}} \bigl( e^{-16(u+v)^2/(5x)} -
e^{-16(u-v)^2/(5x)}\bigr)\\
&&\qquad\quad{} - (1/2) (u-v) \bigl(2
\Phi\bigl(4(u-v)/\sqrt{5x}\bigr) -1\bigr) \\
&&\qquad\quad{} + (1/2) (u+v) \bigl(2
\Phi\bigl(4(u+v)/\sqrt{5x}\bigr)-1\bigr).
\end{eqnarray*}
\end{theorem}
\begin{remark} The above formula can be slightly simplified, but
we leave it in the present form to show that the expression is
symmetric in $u$ and $v$. This is evident once we recall that
$2\Phi(a) -1$ is an odd function. Symmetry in $u$ and $v$ is
something that we expect because of the invariance of $X$ under
the symmetry with respect to the horizontal axis and invariance
under vertical shifts. Note that the formula does not depend
only on $|u-v|$. This is because $X$ is not invariant under
horizontal shifts. The reason is that $X(t) = 0$, a.s., for
every $t$ of the
form $(0, y) \in\RR(\GG_\rho)$; this is not true for any other vertex.
\end{remark}
\begin{pf*}{Proof of Theorem~\ref{p:n14.1}}
The core of our argument is based on the harness idea, just like
the arguments in Section~\ref{sec:harn}.

In this proof we will distinguish between points in the representation
of a graph and their projections on the real axis. So far, this
distinction was not very helpful, so it was ignored in most of the
paper. We will identify edges $E_{jk}$ with sets $E_{jk} ([t_j, t_k])
\subset\R^2$. Recall the following convention introduced after
Definition~\ref{def:2.4}:
$\ol t_j = (t_j , E_{jk}(t_j))$. As a first application of this
notation, we write $\ol v_\rho= (v_\rho, x_\rho)$ and $\ol u_\rho=
(u_\rho, 0_\rho)$. The meaning of $X(\ol t)$ for $\ol t \in\GG_\rho$
is clear.

Recall that hexagonal cells in $\RR(\GG_\rho)$ have diameter
$\rho$, and one of the vertices is located at $(0,0)$. Let
$\rho_1 = \rho\sqrt{3}/4$.

We will construct a tower of finite NCC graphs.
Let $\Gamma_1 \subset\RR(\GG_\rho)$ be the graph of a nondecreasing
function with a starting point $(y_1,0)$ on the horizontal axis and
endpoint at $\ol v_\rho$. It is easy to check that this defines
$\Gamma
_1$ uniquely. Similarly, let $\Gamma_2 \subset\RR(\GG_\rho)$ be the
graph of a nonincreasing function with the starting point at $\ol
v_\rho$ and an endpoint $(y_2,0)$ on the horizontal axis. Let
$\Gamma_3 \subset\RR(\GG_\rho)$ be the graph of a function on the
interval $(-\infty, \infty)$ with values in $[-\rho_1, 0]$; such a
function is unique. Let $\Gamma_4 = \Gamma_1 \cup\Gamma_2 \cup
\Gamma
_3$ and note that $\RR(\GG_\rho) \setminus\Gamma_4$ has two unbounded
connected components, say, $\Gamma_5$ and $\Gamma_6$. Let $\Gamma_7 =
\RR(\GG_\rho) \setminus(\Gamma_5 \cup\Gamma_6)$.
It is easy to see that $\Gamma_7$ is a representation of a TLG $\GG_*$.

All vertices of $\RR(\GG_\rho)$ lie
on lines $L_j:=\{(t,x)\dvtx x= j \rho_1\}$. Let $\WW_j$ be the set
of all vertices of $\RR(\GG_\rho)$ that lie on $\bigcup_{n\leq
j} L_n$. Let $\GG_j=(\VV_j,\EE_j)$ be the graph obtained from
$\GG_*$ by deleting all vertices in $\RR(\GG_*)$ which are not in
$\WW
_j$ and
all corresponding edges.
Note that $\RR(\GG_j) \cap L_j \ne\varnothing$ but there are no vertices
in $\VV_j$ with representation in $L_j$.

It is easy to see that $\GG_j$ are elements of a tower of NCC graphs
that starts from a graph with a single full time path represented by
$\Gamma_3$ and ends with $\GG_*$. To construct such a tower, we add
edges, one at a time, at the top layer of $\GG_j$, until we obtain
$\GG_{j+1}$.
A similar idea can be used to continue the construction of the tower
beyond $\GG_*$, so that the union of all the graphs in the tower is
$\GG
_\rho$.
This construction, Theorem~\ref{m25.3} and Remark~\ref{m25.1}(iv) show
that the restriction of $X$ to $\GG_j$ is a natural Brownian motion on
$\GG_j$.

Note that $x_\rho$ is within distance $\rho_1$ of
$[\sqrt{3}x/(4\rho)]$. We define $j_*$ by $(v_\rho, x_\rho)\in
L_{j_*}$.
Let $\WW^-$ be the set of all vertices of the
form $(s, x) \in\RR(\GG_\rho)$ with $s\leq0$. Let $\HH_k$ be the
$\sigma$-field generated by $\{X_\sigma, \sigma\in\EE_k\}$ and
by $\{X(t), t\in\WW^-\}$. The family $\{\HH_k\}$ is a filtration, that
is, $\HH_k
\subset\HH_{k+1}$ for all $k$. Consider a vertex $\ol t$ of $\GG_*$
whose representation
belongs to $L_{k}$. Recall that $\ol t$ does not belong to $\VV_k$.
Hence, $\ol t$ is in the interior of an edge in $\EE_k$ connecting two
vertices in $\VV_{k}\cap L_{k-1}$. Let $\ol\NN(\ol t)$ denote the
set of
endpoints of this edge and let $\NN(\ol t) $ be the projection of
$\ol\NN(\ol t)$ on the time axis.

Let $\{\beta_u, u \ge0\}$ be a one-dimensional Brownian
motion independent of $X$, starting at $\beta_0 = v_\rho$. We
define a sequence of stopping times, starting with $\tau_0=0$
and
\[
\tau_1 = \inf\{ u \ge0\dvtx\beta_u \in\NN(
\ol v_\rho) \}.
\]
Let $\ol\beta(\tau_1)$ be
the point in $\ol\NN( \ol v_\rho)$ with the time
coordinate $\beta(\tau_1)$ and note that $\ol\beta(\tau_1) \in
L_{j_*-1}$.

If $\beta(\tau_m) \leq0$, then we let $\tau_{m+1} = \tau_m$.
Otherwise we let
\[
\tau_{m+1} = \inf\{ u\ge\tau_m\dvtx
\beta_u \in\NN( \ol\beta_{\tau_m} ) \}.
\]
Let
$\ol\beta(\tau_{m+1})$ be the point in $\ol\NN(
\ol\beta_{\tau_m} )$ with the time coordinate
$\beta(\tau_{m+1})$; then $\ol\beta(\tau_{m+1}) \in L_{j_*-m-1}$. It
follows from our definition of $\GG_*$ that $\ol\NN( \ol
\beta_{\tau_m} )\subset\RR(\GG_*)$. See Figure~\ref{fig3}.

%
\begin{figure}[b]

\includegraphics{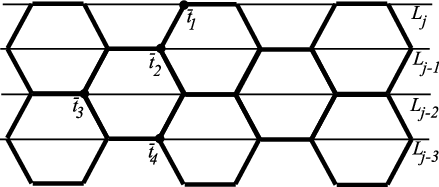}

\caption{The figure represents a fragment of the upper part of $\GG
_j$. The
points $\ol t_k$, $k=1,2,3,4$, represent a possible path of
$\ol\beta_{\tau_m}$, $m= j_*-j, j_*-j+1, \ldots.$}\label{fig3}
\end{figure}

Recall that for a random variable $Z$, $\E_\beta Z$ denotes the
expectation with respect to the law of $\beta$, that is, a
function of $X$.
We have shown that the restriction of $X$ to $\GG_j$ is a unique in law
natural Brownian motion on $\GG_j$.

Let $\ol\NN(\ol v_\rho) = \{\ol t_1,\ol t_2\} \subset\WW_{j_*-1}$,
with $t_1 < t_2$. By the graph-Markovian property of $X$ on the graph
$\GG_{j_*-1}$,
\[
\E( X(\ol v_\rho) \mid\HH_{j_*-1}
) = \E( X(\ol v_\rho) \mid X(\ol t_1), X(\ol t_2) ).
\]
Applying the harness property (\ref{eq:hness2}), we get
\begin{eqnarray*}
\E( X(\ol v_\rho) \mid\HH_{j_*-1})
&=& \frac{v_\rho-t_1}{t_2 - t_1} X(\ol t_2) +
\frac{t_2-v_\rho}{t_2-t_1} X(\ol t_1)\\
&=&\PPP\bigl( \beta(
\tau_1 ) = t_2 \bigr)X(\ol t_2) + \PPP\bigl( \beta(
\tau_1
)= t_1 \bigr)X(\ol t_1) \\
&=&\PPP\bigl( \ol\beta(
\tau_1 ) = \ol t_2 \bigr)X(\ol t_2)
+ \PPP\bigl( \ol\beta( \tau_1
)= \ol t_1 \bigr)X(\ol t_1)
\\
&=& \E_\beta(X(
\ol\beta_{\tau_1} )).
\end{eqnarray*}
We now
proceed by induction. Suppose that
%
%
\begin{equation}\label{eq:n14.3}
\E(
X(\ol v_\rho) \mid\HH_{j_*-m} )= \sum_{\ol t_i \in
\WW_{j_*-m}\cup\WW^-} \PPP\bigl( \ol\beta( \tau_m
)= \ol t_i
\bigr) X(\ol t_i).
\end{equation}
Then, by the the tower property, we get
%
%
\begin{equation}\label{eq:n14.4}
\E( X(\ol v_\rho) \mid\HH_{j_*-m-1}
)= \sum_{\ol t_i \in\WW_{j_*-m} \cup\WW^-} \PPP\bigl(
\ol\beta( \tau_m)= \ol t_i \bigr) \E( X(\ol
t_i) \mid
\HH_{j_*-m-1}).\hspace*{-32pt}
\end{equation}
If $\ol t_i \in\WW^-$ then $\E(
X(\ol t_i) \mid\HH_{j_*-m-1}) = X(\ol t_i)$. If $\ol t_i \in
\WW_{j_*-m}\setminus\WW^-$ then let $\ol\NN(\ol t_i) =
\{\ol t_{i_1},\ol t_{i_2}\} \subset\WW_{j_*-m-1}$. From the
graph-Markovian property and
harness property applied to $X$ restricted to $\GG_{j_*-m}$,
\begin{eqnarray*}
\E( X(\ol t_i) \mid\HH_{j_*-m-1} ) &=&
\E( X(\ol t_{i})\mid X(\ol t_{i_1}), X(\ol t_{i_2}) )\\
&=&
\PPP\bigl( \ol\beta( \tau_{m+1} )= \ol t_{i_1} \mid
\ol\beta(\tau_m)=\ol t_i \bigr) X(\ol t_{i_1}) \\
&&{} + \PPP\bigl( \ol
\beta(\tau_{m+1} )=\ol t_{i_2} \mid\ol\beta(\tau_m)=\ol t_i
\bigr)X(\ol t_{i_2}).
\end{eqnarray*}
Substituting this expression
back in (\ref{eq:n14.3}) and (\ref{eq:n14.4}) we get
\[
\E(
X(\ol v_\rho) \mid\HH_{j_*-m-1} ) = \sum_{\ol t_i \in
\WW_{j_*-m-1}\cup\WW^-} \PPP\bigl( \ol\beta( \tau
_{m+1})=
\ol t_i \bigr) X(t_i) =
\E_\beta(X(\ol\beta_{\tau_{m+1}})).
\]
We are
interested in the case when $j_*-m-1 =0$, that is,
%
%
\begin{equation}\label{eq:n15.1} \E( X(\ol v_\rho) \mid
\HH_{0} ) = \E_\beta(X(\ol\beta_{\tau_{j_*}})).
\end{equation}

For integer $m\geq1$, let $\ol B_{m \rho^2} := \ol\beta_{\tau_{m}}$
and define $\ol B_s$ for other values of $s\geq0$ by $\ol B_s =
\ol B_{[s/\rho^2]\rho^2}$. Let $\ol C_{k} = \ol B_{(k+1)\rho^2} -
\ol B_{k\rho^2}$.
Let $B_s$ be the projection of $\ol B_s$ on the time axis, and
similarly, let $C_k$ be the projection of $\ol C_k$ on the time axis.
Let us ignore for the moment the possibility that
$\ol\beta$ hits $\WW^-$. The random variables $C_k$ are not
independent but they form a Markov chain. By~\cite{B},
Example 2, pa\-ge~167, and \cite{B}, Theorem 20.1, when $\rho\to0$, the
process $B_s =\sum_{k\leq[s/\rho^2]-1} C_{k}$ converges weakly
in the Skorokhod space to Brownian motion $W_s$ with a diffusion
coefficient~$\mathbf s$. We will next calculate $\mathbf s$.

The possible values of $C_k$'s are $-3\rho/4,-\rho/4, \rho/4$
and $3\rho/4$. If we list all states of $C_k$ in this order then
the transition matrix for this Markov chain is
\[
\pmatrix{
1/4 & 0 & 3/4 & 0 \cr1/4 & 0 & 3/4 &
0 \cr0 & 3/4 & 0 & 1/4 \cr0 & 3/4 & 0 & 1/4 }.
\]
The stationary distribution for $C_k$ is
$(1/8, 3/8, 3/8,1/8)$. Hence, in the stationary regime,
\begin{eqnarray*}
\E C_k^2 &=& (-3\rho/4)^2(1/8) + (-\rho/4)^2(3/8) +
(\rho/4)^2(3/8) + (3\rho/4)^2(1/8) \\
&=& 5\rho^2/32.
\end{eqnarray*}
Since the process $B_{k\rho^2}$ is a martingale, it follows that
\[
\var(B_{k\rho^2}-B_0) = \sum_{0\leq n\leq k-1}
\var C_n = \sum_{0\leq n\leq k-1} 5\rho^2/32 = k 5\rho^2/32.
\]
Hence, the diffusion coefficient $\mathbf s$ of $W_s$
is $\sqrt{5/32}$.

Note that although we suppressed $\rho$ in the notation for the process
$B_s$, the distribution of this process depends on $\rho$. Recall that
$\rho_1 = \rho\sqrt{3}/4$, and let $A^\rho_s =x_\rho-s\rho _1$ for
$s\geq0$. Heuristically, $\ol\beta_{\tau_m} = (B_{m\rho^2},A^\rho
_{m})$, $m\in\Z$, $m\geq0$, is a space--time discrete time Markov
chain, with the ``time'' $A^\rho_{m}$ running in the negative direction
along the second axis, starting from $x_\rho$, and the speed of
$\rho_1$ per one step. The ``space'' component $B_{m\rho^2}$ of this
process runs along the first axis, starting from~$v_\rho$. The
right-hand side of (\ref{eq:n15.1}) is evaluated by integrating the
values of $X$ with respect to the hitting distribution of $\WW_0 \cup
\WW^-$ by $(B_{m\rho^2},A^\rho_{m})$. Since $X(\ol t) =0$ for $\ol t$
of the form $(0,y)$, $u_\rho>0$ and $X(\ol\beta_{\tau_{j_*}})\leq0$, we
obtain from the graph-Markovian property $\E(X(\ol u_\rho)
X(\ol\beta_{\tau_{j_*}}) )=0$. For $\ol t_1, \ol t_2 \in L_0$,
$\E(X(\ol t_1) X(\ol t_2)) = t_1 \land t_2$.

Let $A_s =x-s$ for $s\geq0$. When $\rho\to0$, processes $(B_{s\rho
^2},\rho_1 A^\rho_{s})$, $s\geq0$, converge to space--time
Brownian motion $(W_s, A_s)$, $s\geq0$, with $(W_0,A_0) = (v,x)$,
stopped at the exit time from the first quadrant, with the ``time''
component $A_s$ running at
the standard speed and the spatial component having diffusion
coefficient ${\mathbf s} = \sqrt{5/32}$. Let $\tau^* $ be the exit time
from the first quadrant by $(W_s, A_s)$, and let $\tau^{**}$ be the
exit time from the upper half-plane.
Let $K$ be the vertical part of the boundary of the first quadrant.
Let $f$ be the real valued function defined on the boundary of the
first quadrant, with zero values on $K$ and such that $f(t,0) = u \land
t$ for $t\geq0$. Then, by weak convergence, and using (\ref
{eq:n15.1}), $\lim_{\rho\to0} \E(X(\ol u_\rho) X(\ol v_\rho)) =
\E
f(A_{\tau^*}, W_{\tau^*})$.

Let $\phi(s)$ be the density of normal random variable with mean
$v$ and variance $5x/32$, that is,
\[
\phi(s) = \frac1
{\sqrt{5 \pi x/16}} \exp\biggl(- \frac{(s-v)^2}{5x/16}\biggr).
\]
Then using the reflection principle at the hitting time of $K$ we obtain
\begin{eqnarray*}
&&\lim_{\rho\to0} \E(X(\ol u_\rho) X(\ol v_\rho)) \\
&&\qquad= \E f(A_{\tau^*},
W_{\tau^*}) \\
&&\qquad= \E f(A_{\tau^{**}}, W_{\tau^{**}})
- \E\bigl(f(A_{\tau^{**}}, W_{\tau^{**}}) \bone_{\{(A_{\tau^*},
W_{\tau
^*}) \in K\}} \bigr)\\
&&\qquad= \biggl(\int_0^u s\phi(s) \,ds +
\int_u^\infty u \phi(s) \,ds \biggr)
- \biggl(\int_{-u}^0
(-s)\phi(s) \,ds + \int_{-\infty}^{-u} u \phi(s) \,ds
\biggr)\\
&&\qquad= \int_{-u}^u s\phi(s) \,ds + \int_u^\infty
u \phi(s) \,ds - \int_{-\infty}^{-u} u \phi(s) \,ds\\
&&\qquad= \frac
{\sqrt{5x}} {8\sqrt{\pi}} \bigl( e^{-16(u+v)^2/(5x)} -
e^{-16(u-v)^2/(5x)}\bigr)\\
&&\qquad\quad{} - (1/2) (u-v) \bigl(2
\Phi\bigl(4(u-v)/\sqrt{5x}\bigr) -1\bigr)\\
&&\qquad\quad{} +(1/2) (u+v) \bigl(2
\Phi\bigl(4(u+v)/\sqrt{5x}\bigr)-1\bigr).
\end{eqnarray*}
\upqed\end{pf*}

\section*{Acknowledgment}
We are grateful to the referee for a very careful reading of the
original manuscript and very helpful suggestions for
improvement.

%

%
\printaddresses

\end{document}